\documentclass[11pt]{article}

\usepackage{amsmath}
\usepackage{epsfig}
\usepackage{amsthm}
\usepackage{graphicx}
\usepackage{color,multirow,array}
\usepackage{bbm}
\usepackage[matrix,frame,arrow]{xy}

\usepackage{dcolumn}
\usepackage{bm}

\usepackage{amssymb}
\usepackage{bbm}
\usepackage{setspace}
\usepackage{url}

\newtheorem{lemma}{Lemma}[section]
\newtheorem{theorem}[lemma]{Theorem}
\newtheorem{corollary}[lemma]{Corollary}
\newtheorem{proposition}[lemma]{Proposition}

\theoremstyle{definition}
\newtheorem{definition}[lemma]{Definition}

\theoremstyle{definition}
\newtheorem{observation}[lemma]{Observation}
\theoremstyle{definition}
\newtheorem{assumption}[lemma]{Assumption}
\theoremstyle{definition}
\newtheorem{example}[lemma]{Example}

\theoremstyle{definition}
\newtheorem{algorithm}[lemma]{Algorithm}

\theoremstyle{definition}
\newtheorem{construction}[lemma]{Construction}

\long\def\symbolfootnote[#1]#2{\begingroup%
\def\thefootnote{\fnsymbol{footnote}}\footnote[#1]{#2}\endgroup}

\parskip=\medskipamount                 
\thicklines                         

\def\qd{{\kern0.5pt
                   q \kern-5.05pt \raise5.8pt\hbox{$\textstyle.$}\kern
0.5pt}}

\makeatletter
\newtheorem*{rep@theorem}{\rep@title}
\newcommand{\newreptheorem}[2]{%
\newenvironment{rep#1}[1]{%
 \def\rep@title{#2 \ref{##1}}%
 \begin{rep@theorem}}%
 {\end{rep@theorem}}}
\makeatother

\newreptheorem{theorem}{Theorem}
\newreptheorem{lemma}{Lemma}
\newreptheorem{definition}{Definition}

\begin{document}
 
\title{The Weisfeiler-Lehman Method\\ and Graph Isomorphism Testing}
\author{Brendan L. Douglas}

\date{}

\bibliographystyle{plain}

\maketitle

\begin{abstract}

Properties of the `$k$-equivalent' graph families constructed in Cai, F\"{u}rer and Immerman \cite{CFI} and Evdokimov and Ponomarenko \cite{Evdokimov99} are analysed relative the the recursive $k$-dim WL method. An extension to the recursive $k$-dim WL method is presented that is shown to efficiently characterise all such types of `counterexample' graphs, under certain assumptions. These assumptions are shown to hold in all known cases.

\end{abstract}
\section*{Introduction}
In this paper the application of the Weisfeiler-Lehman (WL) method to the graph isomorphism (GI) problem is considered. Following its introduction in the 1970's, this method was considered to be a possible candidate for a solution to the GI problem. However subsequent seminal papers by Cai, F\"{u}rer and Immerman \cite{CFI}, and Evdokimov and Ponomarenko \cite{Evdokimov99} seemed to eliminate this consideration, presenting families of non-isomorphic pairs of graphs which the WL method cannot distinguish in polynomial time, relative to graph size. Indeed, following the work of \cite{CFI}, the question of whether the WL method or some minor variation might solve GI has (to the knowledge of the author) been considered closed. However by analysing the effects of a slight variant of the WL method presented in these works, this paper intends to re-open the question as to whether the general WL approach might be used to solve the GI problem.

In this work we focus on partitioning the vertex set of a given graph to its orbits, rather than on directly providing a certificate characterising the graph's isomorphism class. In particular, whilst the counterexample graph pairs of \cite{CFI, Evdokimov99} cannot be distinguished by the $k$-dim WL method, we consider the case of characterising individual graphs using the recursive $k$-dim WL method.

It is proven that the graphs CFI($G$) and X($G$) (the counterexample graphs constructed by \cite{CFI} and \cite{Evdokimov99} respectively) will be individually characterised by the recursive $(k+1)$-dim WL method, provided that the original graph $G$ is characterised by the recursive $k$-dim WL method.

Hence the direct graph types constructed in \cite{CFI,Evdokimov99} do not necessarily provide counterexamples to the recursive WL method.  Of course directly addressing the recursive $k$-dim WL method was not the purpose of either of these works, and by itself this does not constitute a significant result, in that trivial extensions of these graph can be constructed which provably \emph{do} constitute counterexamples\footnote{For instance, the join of CFI($G$) and CFI$^{'}(G)$ will trivially not be distinguished by the recursive $k$-dim WL method in the case where CFI($G$) is $k$-similar - i.e. when $G$ has no separator of size $k$.}.

However, this result does become significant when combined with the decomposition method of Chapter \ref{sec:decomposition}, in which it is proven that such composite graphs that are also counterexamples to the recursive $k$-dim WL method and additionally satisfy the assumptions of Section \ref{sec:assumptions}, will be characterised by an extended $k$-dim WL method which includes this decomposition process.

In particular, although the distinction between using the WL method to either directly produce a graph certificate or to refine the vertex set to its orbits may seem at first to be a trivial one, since the inability to produce a certificate for individual graphs implies the inability to partition certain combinations of these graphs down to their orbits, we analyse the conditions under which graphs derived from these counterexample families have been shown to not be efficiently partitioned down to their orbits by the WL method. We show that in the cases where this trivially occurs, the graphs of interest possess certain restrictive properties, allowing an extension to the WL method that includes the decomposition method of Chapter \ref{sec:decomposition} to partition these graphs down to their orbits, and thus allowing the recursive WL method to distinguish them.

Part of the significance of these results is that there are no longer any known counterexamples to this extended WL method. The constraints imposed on the initial composite graphs to facilitate the proofs of Chapter \ref{sec:decomposition} do not seem particularly onerous, in the sense that it does not appear easy to circumvent them, finding graphs for which they are not satisfied. It is the focus of future efforts to investigate general properties of the `counterexample' $k$-equivalent graphs, together with attempting to impose further constraints on graphs for which the decomposition method of Section \ref{sec:decomposition} does not apply. Finding such graphs would represent a true advance in the study of the WL method, as they must possess novel properties, presumably intimately related to the property of $k$-equivalence.

Note that far from providing convincing arguments that the Weisfeiler-Lehman method solves GI, the aim of this paper is merely to argue that this 


The structure of this paper is as follows: Chapter \ref{sec:GIprob_back} provide somes background to the graph isomorphism problem in general, and the Weisfeiler-Lehman method in particular. A formal description of the WL method is then provided in Section \ref{sec:formal}, discussing and deriving some basic properties of this method, and in this context highlighting and discussing some of the major counterexample results of \cite{CFI} in Section \ref{sec:CFI}. Chapter \ref{sec:schemes} provides a brief background to coherent configurations, discussing the counterexample graphs derived in \cite{Evdokimov99}. Following this introductory material, Chapters \ref{sec:ext}-\ref{sec:decomposition} contain the new results provided by this work. In Chapter \ref{sec:ext} the properties of general graph extensions (of which the graph families of \cite{CFI} and \cite{Evdokimov99} are examples) are explored, and we show that the ability of the recursive $k$-dim WL method to distinguish graphs is invariant under such extensions, given certain assumptions. In Chapter \ref{sec:orbit} the known $k$-equivalent `counterexample' graphs are recast relative to the recursive $k$-dim WL method. In Chapter \ref{sec:WLprops} some relevant properties of the WL method are derived. Finally in Chapter \ref{sec:decomposition} an extension to  the recursive $k$-dim WL method is presented and is shown to successfully characterise all known $k$-equivalent graphs, given certain assumptions, with some implications of these results discussed in Chapter \ref{sec:conclusions}.

\setcounter{section}{-1}
\section{Graph Theoretic Notation}
\label{sec:notation}

For a graph $G$, we denote the vertex and edge sets of $G$ by $V(G)$ and $E(G)$ respectively. $\overline{G}$ represents the complement of $G$, in which edges and non-edges are switched. Let $v \in V(G)$. Then $d(v)$ and $e(v)$ denote the sets of neighbours and non-neighbours of $v$, such that
\begin{align*}
 &d(v) = \{ x \in V(G) : \{v,x\} \in E(G) \}, \;\;\textrm{ and similarly}\\
 &e(v) = \{ x \in V(G) : \{v,x\} \notin E(G) \}.
\end{align*}
The valency of $v$ is the number of edges incident with $v$, namely $|d(v)|$. We will generally be dealing with undirected graphs. Where this is not the case, the neighbour set $d(v)$ includes di-edges incident with $v$ oriented in either direction. If $H$ is a proper subgraph of $G$, denoted $H \subset G$, sets $S$ in (resp. operations on) $G$ whose extent (resp. action) is \emph{restricted} to $H$ will be denoted by $S\hspace{-0.1cm}\mid_H$, or $S$ restricted to $H$. Given some property or value $c$ held by some members of a set $S$, the set of all members of $S$ with the property $c$ is denoted by $[c]$. For instance, the set of all vertices (within some implicit set $S$) with colour class $c$ is $[c]$. Alternatively, when $c$ is a positive integer, $[c]$ denotes the set $\{1,\ldots,c\}$. Where $S \subset V(G)$ for some graph $G$, the subgraph of $G$ induced on $S$ will simply be referred to as $S \subset G$, where the question of whether $S$ denotes a set of vertices or a graph will be clear from the context where not explicitly stated. Similarly, $G \backslash H$ denotes either the relevant induced subgraph or the vertex set, depending on the context. The \emph{distance} between two vertices $u,v \in V(G)$ is defined as the length of the smallest path connecting them, and denoted $\textrm{dist}(x,y)$.

A \emph{separator} of a graph $G$ is a subset $S \subset V(G)$ such that $G \backslash S$ has no connected components of size $|V|/2$ or larger. The \emph{separator size} of $G$ is the size of the smallest such separator. The \emph{join} of two graphs $G_1$ and $G_2$ is the graph $H = G_1 \cup G_2$ in which $V(H) = V(G_1) \cup V(G_2)$ and $E(H) = E(G_1) \cup E(G_2) \cup A$, where
\begin{align}
 A = \{ \{x,y\} \in V(H) \times V(H) : x \in V(G_1), y \in V(G_2) \}.
\end{align}

\setcounter{equation}{0}
\section{The Graph Isomorphism Problem}
\label{sec:GIprob_back}

\subsection{A General Background}

The question of efficiently determining whether two given graphs are isomorphic is a long-standing open problem in mathematics. It has attracted considerable attention and effort, due both to its practical importance and its relationship to questions of computational complexity. Examples of excellent references articles providing a more thorough background to the GI problem can be found in \cite{Zemlyachenko85}, \cite{Miller77} and \cite{Read77}.

The exact complexity status of the graph isomorphism (GI) problem remains unknown. It is known to be in the class NP, however neither an NP-completeness proof or a polynomial time solution have been found. It is generally considered unlikely to be in NP-complete \cite{Arvind06, Koebler93}, in part because the corresponding testing and counting problems are polynomial-time equivalent, unlike the apparent case for all other known NP-complete problems. Further supporting evidence is provided by Sch\"{o}ning \cite{Schoning87}, who demonstrates that GI is not NP-complete unless the polynomial-time hierarchy collapses. As such it provides a promising candidate for a problem that is neither in P nor NP-complete.

Efficient GI algorithms do exist for several restricted classes of graphs, such as trees \cite{Colbourn81}, planar graphs \cite{Hopcroft74} and graphs with certain bounded parameters, including valence \cite{Luks80}, eigenvalue multiplicity \cite{Babai82} and genus \cite{Miller80}. The GI problem has several additional, relevant properties. It is generally easy to solve in practice, and for many, if not most practical applications the GI problem can be viewed as solved, in that the types of graphs involved can be efficiently characterised by existing algorithms, such as Brendan Mackay's `Nauty' package \cite{McKay81}. It is also easy to solve for almost all graphs \cite{Babai80, Babai79}. However the best current GI algorithm for general graphs has an upper bound of O($\textrm{e}^{\sqrt{n \textrm{log} n}}$) \cite{Babai83, Zemlyachenko85, Babai84}. Hence the interest in GI lies largely in its complexity status, it being one of the interesting problems where practical and theoretical notions of efficiency do not coincide.

GI is polynomial time equivalent to several related problems, including finding an isomorphism map between graphs, if it exists, and determining either the order, generators or orbits of the automorphism group of a graph \cite{Mathon79}. Proposed algorithms to distinguish graphs generally fall into two main (not necessarily disjoint) categories: combinatorial and group theoretic. Here we will be discussing a common type of combinatorial method, based on the iterative vertex-classification (or vertex-refinement) method, and collectively termed the Weisfeiler-Lehman method.

\subsection{The Weisfeiler-Lehman method}

The general type of method labelled as iterative vertex classification is discussed in \cite{Read77} and \cite{Zemlyachenko85}. 

Perhaps the simplest such method begins by partitioning the vertex set of a graph (or equivalently colouring the vertices) according to vertex valency. Then at each subsequent step the colour of each vertex is updated to reflect its previous colour together with the multiset of colours of its neighbours. This proceeds iteratively until a stable colouring (or equitable partition) is reached. This method is also known as the 1-dimensional Weisfeiler-Lehman method. The history and details of the generalised $k$-dimensional Weisfeiler-Lehman method (which we will term the $k$-dim WL as in \cite{CFI} and \cite{Pikhurko10}) and other related methods can be found in \cite{Weisfeiler76, WL68, Friedland89}, among others. Although conceptually quite simple, the 1-dim WL method succeeds in characterising almost all graphs in linear time \cite{Babai79}, although it cannot for instance partition the vertex set of regular graphs. 

In the $k$-dim WL method, we instead start with $k$-tuples of the vertex set, colouring them according to their isomorphism type. At each step the set of $k$-tuples is further partitioned by considering the ordered multiset of colours of the `neighbours' of a given $k$-tuple (here the neighbours are the $k$-tuples differing in exactly one element). Again, this is repeated until an equitable partition is reached. Following its introduction in \cite{WL68}, the general $k$-dim WL method and related methods have reappeared several times, and in various forms. For instance Audenaert \emph{et al}. \cite{Audenaert07} proposed a graph isomorphism method based on the symmetric powers of the adjacency matrix of a graph, which was later shown in \cite{Alzaga10} and \cite{Barghi09} to be no more effective than the $k$-dim WL method. Similarly, GI algorithms based on quantum walks have been proposed in \cite{Douglas08} and \cite{Gamble10}. The algorithm of \cite{Gamble10} was shown in \cite{Smith10} to be no stronger than the $k$-dim WL, while a corollary of the discussion here is that the algorithm of \cite{Douglas08} is trivially no stronger than a variant of the $k$-dim WL method, which we term the depth-$(k-1)$ 1-dim WL method, and define in Chapter \ref{sec:formal}. One significant aspect of the WL method alluded to by its continuing reappearance in varying forms is its intuitive, in some ways natural, combinatorial form.

A more formal definition is given in Chapter \ref{sec:formal}, however at this point it is clear that just as the 1-dim WL fails for regular graphs, providing no useful information beyond the graph's order, the 2-dim WL will fail for strongly regular graphs. Similarly, the $k$-dim WL method cannot partition the vertex set of $k$-isoregular graphs (alternatively $k$-tuple regular graphs), defined as in \cite{Cameron80} to be graphs in which the number of common neighbours of any $k$-tuple of a given isomorphism type is constant (for instance, strongly regular graphs are 2-isoregular). The results of \cite{Golfand78} and \cite{Cameron80}, classifying 5-isoregular graphs to a few trivial cases, and proving that 5-isoregular graphs are $k$-isoregular for all $k$, in part supported the conjecture that the $k$-dim WL method might, with $k$ some small constant, suffice to classify all graphs.

The $k$-dim WL method can be implemented for a graph on $n$ vertices in time $O(n^{k+1})$, hence if even the $O(\textrm{log}(n))$-dim WL method sufficed to distinguish all graphs on $n$ vertices it would solve GI. However the results of \cite{Furer87, CFI} disposed of this possibility, providing examples of a family of pairs of graphs with $O(n)$ vertices which the $(n-1)$-dim WL method failed to distinguish. This situation was explored further in \cite{Evdokimov99} and \cite{Barghi09}, in terms of coherent configurations, with an additional family of counterexample graphs presented.

\setcounter{equation}{0}
\section{Formal description of WL method}
\label{sec:formal}

Let $G$ be an edge- and vertex-coloured graph, where $|V(G)| = n$, and for any $u, v \in V(G)$ such that $(u,v) \in E(G)$, $\omega(u)$ and $\omega(u,v)$ denote the colouring of the vertex $u$ and edge $(u,v)$ respectively. Consider the set $V(G)^k$ of $k$-tuples of $V(G)$. The $k$-dim WL method proceeds iteratively, with the colour of all $k$-tuples being updated at each step. Given an ordered $k$-tuple $S = (v_1, v_2, \ldots, v_k) \in V(G)^k$, consider the ordered set of `neighbouring' $k$-tuples 
\begin{align}
\label{eq:S'}
 S'(x) = ((v_1,\ldots,v_{k-1},x),\ldots,(x,v_2,\ldots,v_k)), \;\; x \in V(G).
\end{align}
Then after $t$ steps of the $k$-dim WL method, the colour of $S \in V(G)^k$ is denoted by $\textrm{WL}^{t}_k(S)$, such that
\begin{align}
\label{eq:WL_k}
 &\textrm{WL}^{0}_k(S) = \textrm{ iso}(S), \: \textrm{ and} \notag\\
 &\textrm{WL}^{t}_k(S) = \langle \, \textrm{WL}^{t-1}_k(S), \:\textrm{Sort}\{\, \textrm{WL}^{t-1}_k(S'(x)) : x \in V(G)\} \: \rangle,
\end{align}
where `iso($S$)' denotes the isomorphism class of the ordered $k$-tuple $S$, such that for $S_1 = (x_1,\ldots,x_k)$, $S_2 = (y_1,\ldots,y_k) \in V^k$, we have $\textrm{iso}(S_1) = \textrm{iso}(S_2)$ if and only if for all $i,j \in [k]$ the following hold:
\begin{enumerate}
 \item $x_i = x_j\,$ if and only if $\,y_i = y_j$.
 \item $(x_i,x_j) \in E(G)$ if and only if $(y_i,y_j) \in E(G)$ and $\omega(x_i,x_j) = \omega(y_i,y_j)$.
 \item $\omega(x_i) = \omega(y_i)$.
\end{enumerate}
Note that the sorting function `Sort' used here applies only to the outermost dimension of nested lists, unless otherwise stated. In particular, it does not alter the internal ordering of each individual ordered set comprising $S'(x)$. As an additional note regarding notation, the angle brackets enclosing the right hand side of (\ref{eq:WL_k}) are used in the style of \cite{CFI} to delimit an \emph{ordered} set. They are used here to take the place of round brackets (which denote ordered sets elsewhere in this work) simply for aesthetic purposes, and this convention will be continued when describing $k$-dim WL colour classes as above\footnote{Specifically, this notational convention will only be used to \emph{enclose} the definition (or reference to the definition) of such colour classes.}.

At each step in the process the $\textrm{WL}^{t}_k(S)$ multisets are sorted lexicographically then assigned a number from 1 to $n$ denoting the new colour class of $S$, together with a decoding table to store the remaining information for the purposes of constructing a certificate for the graph at the end. The algorithm stops when the colouring of $k$-tuples is stable; when a further iteration of the method does not partition the set of $k$-tuples further. Let this occur after $r$ steps, such that for all $S_1, S_2 \in V(G)^k$,
\begin{align}
 \textrm{WL}^{r+1}_k(S_1) = \textrm{WL}^{r+1}_k(S_2) \;\;\textrm{ if and only if }\;\; \textrm{WL}^{r}_k(S_1) = \textrm{WL}^{r}_k(S_2).
\end{align}
Then the final colouring of $k$-tuples is denoted $\textrm{WL}^{\infty}_k(S)$, or simply $\textrm{WL}_{k}(S)$, such that
\begin{align}
 \label{eq:WL_k_infty}
 \textrm{WL}_k(S) = \langle \:\textrm{Sort}\{\, \textrm{WL}^{r}_k(S'(x)) : x \in V(G)\} \: \rangle
\end{align}
Hence for graphs $G$ and $H$, $\textrm{WL}_k(G) = \textrm{WL}_k(H)$ if and only if there exists a bijection mapping $V(G)$ to $V(H)$ preserving the colouring of $k$-tuples.

Note that given some constant $k$ and a graph on $n$ vertices, this stable colouring (also known as the equitable partition) will be reached in O(poly($n$)) time. Hence if the $k$-dim WL method succeeded in partitioning all graphs down to their orbits (for some constant or slowly growing $k$), it would solve the GI problem.

Given the stable partitioning of $V^k$, a corresponding partitioning of $t$-tuples, for any $t < k$ can be constructed, such that two $t$-tuples are assigned identical colours if and only if they cannot possibly be distinguished based only on the colouring of $k$-tuples. The process for colouring the $t$-tuples $\mathbf{x} = (x_1,\ldots,x_t)$ as $\textrm{WL}_k(\mathbf{x})$ is detailed in Chapter \ref{sec:WLprops}, in which the following recursive relation is derived:
\begin{align}
 \textrm{WL}_k(\mathbf{x}) = \langle \:\textrm{Sort}\{ \, \textrm{WL}_k(\mathbf{x},i) : i \in V(G)\} \:\rangle,
\end{align}
where $(\mathbf{x},i)$ denotes the ordered $(t+1)$-tuple $(x_1,\ldots,x_t,i)$. Similarly, a colouring of $t$-tuples for $t > k$ can be constructed, again defined recursively by considering the ordered set of $(t-1)$-tuples contained in the $t$-tuple of interest.

In terms of notation, where the $t$-tuple $\mathbf{x}$ belongs to both $G$ and some subgraph of interest $H \subset G$, the colouring $\textrm{WL}_k(\mathbf{x})$ may refer to the colouring of the $t$-tuple within either $G$ or the induced subgraph $H$. Which it refers to will either be clear from the context or specified by the notation $\textrm{WL}_k(\mathbf{x}) \hspace{-0.1cm}\mid_G$ or $\textrm{WL}_k(\mathbf{x}) \hspace{-0.1cm}\mid_H$, meaning the $t$-tuple colouring is \emph{relative to} the $k$-tuple colourings of $G$ or $H$ respectively.

Several closely related variants on the $k$-dim WL method have been proposed. One such variant, appearing in \cite{Weisfeiler76}, employs what could be described as a depth-first approach to the WL method. It is introduced under the umbrella term of `deep stabilisation' in \cite{Weisfeiler76}, and involves stabilising a $k$-tuple followed by applying the 1-dim WL method, then cycling over all possible such $k$-tuples. We will term this method the depth-$k$ 1-dim WL method, and note briefly that it is in some ways analogous to the $(k+1)$-dim WL method, and might be expected to have similar refinement power. Indeed, all relevant results here regarding the $k$-dim WL method can be extended to the depth-$(k-1)$ 1-dim WL method.

As mentioned, for some time the $k$-dim WL method, with sufficiently small $k$ (e.g. where $k = $ O(log($n$)) or even where $k$ is a constant) was thought to represent a potential candidate for the solution to the GI problem. Then Cai, F\"{u}rer and Immerman \cite{CFI} introduced a family of counterexample graphs (we will term them the CFI counterexamples) for which the entire global properties of the graph could not be characterised using a sufficiently low dimension WL method. Specifically, they constructed pairs of non-isomorphic graphs on O($n$) vertices that were distinguished by the $n$-dim WL method, but not by the $(n-1)$-dim WL method.

\subsection{CFI Counterexamples}
\label{sec:CFI}

In the work of \cite{CFI}, pairs of non-isomorphic graphs with O($k$) vertices which cannot be characterised using the $k$-dim WL method were constructed from graphs with separator size $k+1$. In particular, given a graph $G$ they define a related graph $X(G)$ (to be termed CFI($G$) for the remainder of this work), in which each vertex $v \in V(G)$ of valency $k$ is replaced by the graph $\textrm{CFI}(v)$, defined (as in \cite{CFI}) by the relations:
\begin{align}
 V(\textrm{CFI}(x)) = & A(v) \cup B(v) \cup M(v), \textrm{ where } A(v) = \{a_i : 1 \le i \le k \}, \notag\\
 & B(v) = \{b_i : 1 \le i \le k\}, \textrm{ and}  \notag\\
 & M(v) = \{m_S : S \subseteq \{1, \ldots, k\}, |S| \textrm{ is even}\} \\
 E(\textrm{CFI}(x)) = & \{(m_S,a_i) : i \in S\} \cup \{(m_S,b_i) : i \notin S\}  \notag
\end{align}
The middle vertices $M(v)$ of $\textrm{CFI}(v)$ are coloured differently to the other vertices ($A(v) \cup B(v)$). Hence each vertex $v$ of degree $k$ is replaced by a graph of size $2^{k-1}+2k$, consisting of a `middle' section (the $M(v)$ vertices) of size $2^{k-1}$, and $k$ $\{a_i,b_i\}$ pairs of vertices representing the endpoints of each edge incident with $v$, such that each $\{a_i,b_i\}$ pair, $1 \le i \le k$, is associated with some edge $\{v,u\}$ ($u \in V$) incident with $v$. Furthermore, for each edge $\{u,v\} \in E(G)$, the $(a_i,b_i)$ pairs associated with the endpoints of $\{u,v\}$, termed $\{a_{u,v},b_{u,v}\}$ and $\{a_{v,u},b_{v,u}\}$ respectively, are connected, such that $a_{u,v}$ is connected to $a_{v,u}$, and $b_{u,v}$ is connected to $b_{v,u}$.

Similarly, a graph CFI$'$($G$) is defined as above, however with a \emph{single} $\{a_{u,v},a_{v,u}\}, \{b_{u,v},b_{v,u}\}$ edge pair `twisted', in that these two edges are replaced by the edges $\{a_{u,v},b_{v,u}\}, \{b_{u,v},a_{v,u}\}$. Basic properties of these graphs are discussed in detail in \cite{CFI} (also see \cite{Miyazaki97} and \cite{Pikhurko10} for additional details). Here we will recall some pertinent results.

\begin{lemma}
 $|\textrm{Aut}(\textrm{CFI}_k)| = 2^{k-1}$. Each $g \in \textrm{Aut}(\textrm{CFI}_k)$ corresponds to interchanging $a_i$ and $b_i$ for each $i$ in some subset $S$ of $\{1,2,\ldots,k\}$ of even cardinality.
\end{lemma}

\begin{lemma}
 Given a graph $G$, consider a graph CFI$\,''(G)$ defined as in CFI($G$) and CFI$\,'(G)$ above, however with $t$ edges twisted. Then $\textrm{CFI}\,''(G) \cong \textrm{CFI}(G) \textrm{ iff } t\:$ is even, and CFI$\,''(G) \cong \textrm{CFI}\,'(G) \textrm{ iff } t \textrm{ is odd.}$
\end{lemma}

\begin{corollary}
 $\textrm{CFI}(G) \ncong \textrm{CFI}\,'(G)$.
\end{corollary}

Most importantly, the work of \cite{CFI} proved that given a graph $G$ with separator size $k+1$, the $k$-dim WL method cannot directly distinguish the graphs CFI($G$) and CFI$'$($G$).
Specifically, while performing the $k$-dim WL method on CFI($G$) and CFI$'$($G$), at each step the lexicographically sorted multisets of colours are identical.


A similar family of graphs were constructed in \cite{Evdokimov99}, using different methods. They introduce the term `$k$-equivalent graphs' to describe graphs which the $k$-dim. WL method cannot distinguish. Although this trivially encompasses isomorphic graphs, when the term is used here it will refer solely to cases where at least two non-isomorphic graphs exist that cannot be distinguished by the $k$-dim. WL method. In particular, the following terminology will be used.

\begin{definition}
 A graph $G$ will be termed $k$-equivalent if:
 \begin{itemize}
  \item There exists a graph $H$, $G \ncong H$ such that $\textrm{WL}_k(G) = \textrm{WL}_k(H)$. 
  \item $k$ is the largest integer for which this is true (a $k$-equivalent graph is not $(k-1)$-equivalent, hence cannot be $(k+1)$-equivalent).
 \end{itemize}
\end{definition}

It will be convenient to make one exception to the terminology that a graph is $k$-equivalent only if $k$ is the largest such integer, namely for the definition of a non-$k$-equivalent graph.

\begin{definition}
 A graph $G$ is \emph{non}-$k$-equivalent if there does not exist a graph $H$, $G \ncong H$ such that $\textrm{WL}_k(G) = \textrm{WL}_k(H)$.
\end{definition}

\begin{definition}
 Two $k$-equivalent graphs $G_1$ and $G_2$ are \emph{mutually $k$-equivalent} if $\textrm{WL}_k(G_1) = \textrm{WL}_k(G_2)$.
\end{definition}

Mutually $k$-equivalent graphs are not necessarily non-isomorphic, however they will be assumed to belong to a $k$-equivalence class containing at least two distinct isomorphism classes. Namely, if $G_1$ and $G_2$ are mutually $k$-equivalent, and $G_1 \cong G_2$, then $ \exists \; H \ncong G_1$ such that $H$ and $G_1$ are mutually $k$-equivalent.

A related concept to be employed in the following work is that of $k$-similarity.

\begin{definition}
 Two graphs $G$ and $H$ are termed \emph{$k$-similar} if $\textrm{WL}_k(G) = \textrm{WL}_k(H)$. This is denoted by $G \sim H$ (where the relevant $k$ will be clear from the context).
\end{definition}

Hence mutually $k$-equivalent graphs are $n$-similar for all $n \le k$\footnote{And explicitly are \emph{not} $n$-similar for all $n > k$.}, and isomorphic graphs are $n$-similar for all $n$. Note that these concepts of $k$-similarity and $k$-equivalence were used in \cite{Barghi09} (and implicitly in \cite{Alzaga10}) to demonstrate that any two $k$-similar graphs have identical $k$-th symmetric powers, hence addressing the proposition of \cite{Audenaert07}.

\begin{definition}
 Two subgraphs $S \subset V(G)$, $R \subset V(H)$ satisfy the relation $S \sim_k R$ if and only if $\textrm{WL}_k(S) \mid_G = \textrm{WL}_k(R) \mid_H$. Note that $S \sim_k R$ only if $G \sim_k H$.

 In the limit where $S \sim_k R$ for all $k>0$, $k$-similarity becomes isomorphism, and is denoted by $S \sim R$.
\end{definition}

\setcounter{equation}{0}
\section{Coherent Configurations}
\label{sec:schemes}

As it was first proposed in the work of \cite{WL68}, the Weisfeiler-Lehman method takes the form of a matrix algebra associated with a graph, termed the cellular algebra, and later known as the adjacency algebra (or basis algebra) of a coherent configuration. This chapter will provide a brief background to coherent configurations, introducing a further set of $k$-equivalent graphs based on coherent configurations. For a more thorough background to coherent configurations and their relation to the WL method, see \cite{Cameron03,Friedland89,Evdokimov99,Barghi09}.

\subsection{Definitions}

Let $V$ be a finite set, and $\mathcal{R} = \{R_1,\ldots,R_s\}$ be a partition of $V \times V$, such that each $R_i \in \mathcal{R}$ is a binary relation on $V$.  A \emph{coherent configuration} on $V$ is a pair $\mathcal{C} = (V,\mathcal{R})$ satisfying the following conditions:

\begin{enumerate}
 \item There exists a subset $\mathcal{R}_0$ of $\mathcal{R}$ partitioning the diagonal $\Delta(V)$ of the Cartesian product $V \times V$.
 \item $R_i \in \mathcal{R}$ if and only if its transpose $R^{T}_i$ is in $\mathcal{R}$.
 \item Given $R_i,R_j,R_k \in \mathcal{R}$, for any $(u,v) \in R_k$, the number of points $x \in V$ such that $(u,x) \in R_i$ and $(x,v) \in R_j$ is equal to $c^{k}_{i,j}$, independent of the choice of $(u,v) \in R_k$.
\end{enumerate}

The numbers $c^{k}_{i,j}$ are called the \emph{intersection numbers} of $\mathcal{C}$, and the elements of $V$ and $\mathcal{R}$ are called the \emph{points} and \emph{basis relations} of $\mathcal{C}$ respectively. Similar to adjacency matrices of graphs, a basis relation $R_i$ can be represented in matrix form by the \emph{basis matrix} $A(R_i)$, where:
 \begin{equation}
    A(R_i)_{x,y} = \left\{
    \begin{array}{rl}
     1 & \textrm{if } (x,y) \in R_i\\
     0 & \textrm{otherwise}.
    \end{array} \right.
 \end{equation}
Then the coherent configuration conditions above take the form:
\begin{enumerate}
 \item $\displaystyle\sum\limits_{i=1}^t S_i = \mathbf{1}_{|V|}$, the identity matrix, where $\mathcal{R}_0 = {S_1,\ldots,S_t}.$
 \item If $R_i \in \mathcal{R}$ there exists a relation $R_j \in \mathcal{R}$ such that $A(R_i)^T = A(R_j)$.
 \item For each $i,j \in [s], \;\, A(R_i) \, A(R_j) = \displaystyle\sum\limits_{k=1}^s c^{k}_{i,j} \, A(R_k)$.
\end{enumerate}
The algebra spanned by the $A(R_i)$ is called the \emph{adjacency algebra} or \emph{basis algebra} of the coherent configuration $\mathcal{R}$.

Consider the set of basis relations $\mathcal{R}_0 = \{S_1,\ldots,S_t\}$ such that $(x,y) \in S_i$ only if $x = y$. Note that by condition (1),
\begin{align}
 (x,x) \in R_i \textrm{ if and only if } u = v, \;\forall\; (u,v) \in R_i.
\end{align}
The $t$ sets $F_i \subset V$ such that $F_i = \{ x \in V : (x,x) \in S_i \}$ are called the \emph{fibres} of $\mathcal{C}$. By condition (1) they form a partition of $V$.

\subsection{Weak and Strong Isomorphisms}

Two coherent configurations $\mathcal{C} = (V,\mathcal{R})$ and $\mathcal{C}' = (V',\mathcal{R}')$ are \emph{isomorphic} (or strongly isomorphic) if there is a bijection mapping $V$ to $V'$, preserving the basis relations. The coherent configurations $\mathcal{C}$ and $\mathcal{C}'$ are \emph{similar} (or weakly isomorphic) if there exists a bijection $\varphi : \mathcal{R} \rightarrow \mathcal{R}'$, where $\varphi : R_i \mapsto R_{\varphi(i)}$, such that
\begin{align}
 c^{k}_{i,j} = c^{\varphi(k)}_{\varphi(i),\varphi(j)}, \quad \textrm{for all } i,j,k \in [s].
\end{align}
Clearly, all strong isomorphisms induce weak isomorphisms, however the converse does not hold.

\subsection{Coherent Configurations of Graphs}

The set of coherent configurations on $V$ forms a lattice with respect to inclusion \cite{Cameron03}. In particular, given a set of binary relations $\{S_1,\ldots,S_t\}$, where each $S_i \in V \times V$, denote by $[S_1,\ldots,S_t]$ the smallest coherent configuration $\mathcal{C} = (V,\mathcal{R})$ with the property:
\begin{align}
 &\textrm{For each } S_i, \; \exists \textrm{ a set } \{R_1,\ldots,R_x\} \subset \mathcal{R} \textrm{ such that } S_i = \bigcup_{j=1}^x R_j.
\end{align}
$[\mathcal{S}]$ is also termed the \emph{cellular closure} of the set $\mathcal{S}$ of binary relations.
We define the coherent configuration associated with an uncoloured graph $G$ to be $[G] := [V,E,(V \times V) \backslash E]$, the smallest coherent configuration in which the vertices, edges and non-edges are each partitioned by basis relations. Similarly, for an edge- and vertex-coloured graph $G$, consider the initial binary relations of $G$ to be the sets of vertices (and edges) of each colour, together with the set of non-edges, resulting in a corresponding definition for $[G]$.

In fact, the WL method was originally defined in \cite{WL68} as a way of calculating the adjacency matrix of the coherent configuration associated with a graph. Specifically, consider a coloured graph $G$, in which $c(v)$ denotes the colour of vertex $v \in V(G)$, and $c(u,v)$ the colour of edge $(u,v) \in E(G)$.
\begin{theorem}[\cite{WL68}]
 $G$ is the adjacency matrix of a coherent algebra if and only if $G$ is stable relative to the $1$-dim WL method, in that for all $u,v \in V$, $c(u) = c(v)$ only if $\textrm{WL}_1(u) = \textrm{WL}_1(v)$.
\end{theorem}

Analogous to the conversion of $k$-tuple colourings to 1-tuple colourings described in Chapter \ref{sec:WLprops}, the set of basis relations of a coherent configuration $\mathcal{C} = (V,\mathcal{R})$ induce a colouring of the 1-tuples of $V$, corresponding exactly to the subset $\mathcal{R}_0$ of $\mathcal{R}$ that partitions $\Delta(V)$. Denote this colouring of 1-tuples by $\overline{\mathcal{C}}$, the \emph{closure} of $\mathcal{C}$.

A set $\mathcal{R}$ of binary relations on $V$ is termed \emph{1-closed} if $[\mathcal{R}] = (V,\mathcal{R})$. Similarly a graph is termed \emph{1-closed} if it is stable with respect to the $1$-dim WL method - if the adjacency matrix of the graph correspond to that of a coherent configuration. Strongly regular graphs are trivially 1-closed, as their sets of vertices, edges and non-edges satisfy all conditions of a coherent configuration (equivalently, their vertex sets are not refined by the $1$-dim WL method).

\subsection{$m$-Extended Coherent Configurations}

Given a scheme $\mathcal{C} = (V, \mathcal{R})$, let $\mathcal{C}^{m} = \mathcal{C} \otimes \ldots \otimes \mathcal{C}$ denote the $m$-fold tensor product of $\mathcal{C}$, and $\Delta_m$ denote the diagonal of $V^{m} = V \times \ldots \times V$. Then the $m$-\emph{extension} of $\mathcal{C}$ is defined as:
\begin{align}
 \widehat{\mathcal{C}}^{(m)} = [\mathcal{C}^m,\Delta_m].
\end{align}
An isomorphism $\varphi : \mathcal{C}^{(m)} \rightarrow (\mathcal{C}^{(m)})'$ is termed an $m$-isomorphism mapping $\mathcal{C}$ to $\mathcal{C}'$. A similarity (weak isomorphism) from $\mathcal{C}^{(m)}$ to $(\mathcal{C}^{(m)})'$ is termed an $m$-similarity mapping $\mathcal{C}$ to $\mathcal{C}'$.

$\overline{\mathcal{C}^{(m)}}$ denotes the colouring of 1-tuples of $V$ associated with $\mathcal{C}^{(m)}$, termed the \emph{m-closure} of $\mathcal{C}$. A coherent configuration $\mathcal{C}$ (resp. a set of basis relations $\mathcal{R}$) is termed \emph{m-closed} if $\overline{\mathcal{C}} = \overline{\mathcal{C}^{(m)}}$ (resp. if $\overline{[\mathcal{R}]} = \overline{[\mathcal{R}]^{(m)}}$). Similarly a graph is $m$-closed if it is stable with respect to the $m$-dim WL method.

\begin{theorem}[\cite{Evdokimov99}]
Denote the orbit partition of a graph $G$ by $\mathcal{P}$. Then for some $n$,
\begin{align}
 \overline{[G]} \le \overline{[G]^{(2)}} \le \ldots \le \overline{[G]^{(n)}} = \ldots = \mathcal{P}.
\end{align}
\end{theorem}

In \cite{Evdokimov99}, pairs of non-isomorphic $k$-similar coherent configurations are constructed for all $k$. These coherent configurations are related to cubic graphs with minimum separator size of $3k+1$ or larger, similar to the case for the counterexample graphs of \cite{CFI}.

\subsection{Examples of non-isomorphic $k$-similar coherent configurations}

Here we will give a brief description of the $k$-similar, non-isomorphic coherent configurations constructed in \cite{Evdokimov99}, corresponding closely to the definition given in \cite{Barghi09}. Associated with these will be pairs of $k$-equivalent (edge-coloured) graphs which will be analysed together with the $k$-equivalent graphs of \cite{CFI} in depth in later sections.

Let $G$ by a cubic graph on $s$ points. Define a coherent configuration $\mathcal{C} = (V,\mathcal{R})$ on $4s$ points in the following way.
Denote the vertex set of $G$ by $I = \{1,\ldots,s\}$, and associate with each $i \in I$ a fibre $V_i$ of size 4 in $\mathcal{C}$, such that $\mathcal{C}$ has exactly $s$ fibres, each of size 4. For each $V_i = \{0,1,2,3\}$, let $\mathcal{C}_{V_i}$ be the coherent configuration on 4 points with the three non-reflexive basis relations:
\begin{align}
 &E_1 = \{(0,1),(1,0),(2,3),(3,2)\}, \; E_2 = \{(0,2),(2,0),(1,3),(3,1)\} \textrm{ and} \notag\\
 &E_3 = \{(0,3),(3,0),(2,1),(1,2)\}.
\end{align}
As $V_i$ is a fibre of $\mathcal{C}$, $\mathcal{R}_{i,i}$ contains 4 basis relations, where 
\begin{align}
 \mathcal{R}_{i,j} = \{R \in \mathcal{R} : R \subset V_i \times V_j\}.
\end{align}
For $i \neq j$, let
\begin{equation}
\label{eq:Rij}
 |\mathcal{R}_{i,j}| = \left\{
 \begin{array}{rl}
  2 & \textrm{if } \{i,j\} \in E(G)\\
  1 & \textrm{otherwise}.
 \end{array} \right.
\end{equation}
In the cases where $\{i,j\} \in E(G)$, define $\mathcal{R}_{i,j}$ as follows.

\noindent Assign to each $v \in d(i)$ the numbers $c(i,v) \in \{1,2,3\}$ with the property:
\begin{align}
 &u,v \in d(i) \;\textrm{ such that }\; u \neq v \;\, \textrm{ only if } \;\, c(i,u) \neq c(i,v).
\end{align}
$\mathcal{R}_{i,j}$ consists of two distinct relations, labelled $R_{1,2}$, with the di-edge $(i,j)$ relative to which they are defined left implicit. These relations $R_1, R_2 \in \mathcal{R}_{i,j}$, $\{i,j\} \in E(G)$ are defined as:
\begin{align}
\label{eq:R12}
 &R_1 = \langle c(i,j) \rangle \times \langle c(j,i) \rangle \cup \overline{\langle c(i,j) \rangle} \times \overline{\langle c(j,i) \rangle},\\
 &R_2 = (V_i \times V_j) \backslash R_1,
\end{align}
where $\langle c(i,j) \rangle = \{0,c(i,j)\} \subset V_i$ and $\overline{\langle c(i,j) \rangle} = V_i \backslash \langle c(i,j) \rangle$.

For any cubic graph $G$ the coherent configuration $\mathcal{C} = (V, \mathcal{R})$ described above is called a \emph{Klein scheme} associated with $G$. Further, for each $i \in I$, consider the mapping $\psi_i : \mathcal{R} \rightarrow \mathcal{R}$, where
\begin{equation}
 \psi_i(R) = \left\{
 \begin{array}{rl}
  (V_i \times V_j) \backslash R & \textrm{if } R \in \mathcal{R}_{i,j}, \textrm{ and } j \in d(i)\\
  (V_j \times V_i) \backslash R & \textrm{if } R \in \mathcal{R}_{j,i}, \textrm{ and } j \in d(i)\\
  R & \textrm{otherwise}.
 \end{array} \right.
\end{equation}

\begin{theorem}[\cite{Evdokimov99, Barghi09}]
\label{thm:Klein}
 $\psi_i$ is an involutory weak isomorphism from $\mathcal{C}$ to $\mathcal{C}$ not inducing a strong isomorphism. Further, if $G$ has minimum separator size $l > 3k$, $\psi_i(\mathcal{C})$ are $\mathcal{C}$ are $k$-similar.
\end{theorem}

\subsection{The Associated Graph}
Given a Klein scheme $\mathcal{C} = (V,\mathcal{R})$ associated with some cubic graph $G$, we can define an edge-coloured di-graph $K(G)$ associated in turn with $\mathcal{C}$, in the following manner\footnote{Note that this graph is slightly different to that obtained by a direct conversion of $\mathcal{C}$, in that the relations $\mathcal{R_{i,j}}$, where $\{i,j\} \notin E(G)$ are converted to non-edges.}. 

Let $V = \{V_1,\ldots,V_i\}$ as above, and let $V(K(G)) = V$ and $E(K(G))$ be denoted by $E$. Denote the colour of the di-edge $(x,y) \in E$ by $C(x,y)$. Then given $x \in V_i, y \in V_j$,
\begin{align}
 &(x,y) \notin E \;\textrm{ if and only if }\; i \neq j \;\textrm{ and }\; \{i,j\} \notin E(G).
\end{align}
Further, denote the colour of the di-edge $(x,y) \in E$ by $C(x,y)$. Then
\begin{align}
 &C(u,v) = C(x,y) \textrm{ if and only if } (u,v),(x,y) \in R, \textrm{ for some } R \in \mathcal{R}.
\end{align}
Hence di-edges are assigned the same colour only in the case where they belong to the same basis relation of $\mathcal{C}$.

Note that the colours of $K(G)$ are not considered to possess absolute values in the sense of those of $\textrm{WL}_k(G)$, but rather relative values.
Let $K'(G)$ be defined similarly with respect to $\psi_i(\mathcal{C})$, for any $i \in V(G)$. Then $K'(G) \ncong K(G)$, and the following corollary of Theorem \ref{thm:Klein} holds.
\begin{corollary}
 If $G$ has no separators of size $3k$, then $K'(G)$, $K(G)$ are non-isomorphic, $k$-similar graphs.
\end{corollary}
And hence,
\begin{corollary}
 If $G$ has no separators of size $3k$, $\textrm{WL}_k(K'(G)) = \textrm{WL}_k(K(G))$. Hence if $G$ additionally \emph{has} separators of size $3(k+1)$, then $K(G),K'(G)$ are $k$-equivalent.
\end{corollary}

This second family of $k$-equivalent graphs has several similarities to those of \cite{CFI}. In particular, in both cases the differences between non-isomorphic pairs can be `shifted' around the graph; in the case of CFI$(G)$ and CFI$'(G)$ this involves `twisting' an even number of $(a,b)$ edge pairs as described in \cite{Furer87}; in the case of $K(G)$ and $K'(G)$ this involves applying the $\psi_i$ transformation to an even number of fibres $V_i$\footnote{Note that if $\psi$ is applied to $0$ (mod 2) fibres of $K(G)$ then it preserves the isomorphism class of $K(G)$.}. The basis relations $R_1, R_2 \in \mathcal{R}_{i,j}$, for $\{i,j\} \in E(G)$, are in this way analogous to the $(a,b)$ pairs connecting the gadgets $\textrm{CFI}(i)$ to $\textrm{CFI}(j)$ in the graph CFI($G$). In both cases the key property that the graphs possess is that the separator size of the original $G$ is proportional to the size of the $k$-tuples required to distinguish the non-isomorphic pairs.

\setcounter{equation}{0}
\section{Graph Extensions and the $k$-dim WL Method}
\label{sec:ext}

The purpose of this chapter is to analyse the relative properties of $\textrm{WL}_k(G)$ and $\textrm{WL}_k(G')$, where $G'$ is an extension of the graph $G$ resulting from replacing each of the vertices of $G$ by some gadget, then connecting the gadgets according to a certain set of rules. This analysis is motivated by the form of the known families of counterexample graphs, each involving extensions of this kind applied to expander graphs.

In particular, the following theorems will be proven.

\begin{theorem}
\label{thm:cfi(x)}
 If the recursive $k$-dim WL method succeeds in characterising the graph $G$, then the recursive $(k+1)$-dim WL method succeeds in characterising the graph CFI($G$).
\end{theorem}

\begin{theorem}
\label{thm:kleinWL}
 The recursive $1$-dim WL method succeeds in characterising the graph $K(G)$, associated with a Klein scheme of $G$.
\end{theorem}

Following this, we will show a more general result; namely that if a graph $G$ is extended to some graph $G'$ by replacing each vertex by an unbiased gadget of a certain type, of which the F\"{u}rer gadget relating to CFI($G$) is one such example, then the $k$-dim WL method partitions the \emph{gadgets} of $G'$ into the same relative colour classes as it partitions the \emph{vertices} of $G$. Here the colour class of a gadget refers to the sorted set of colour classes of its constituent vertices.

\begin{definition}
\label{def:unbiased}
 The extension of a graph $G$ formed by replacing each vertex $v \in V(G)$ by some type of gadget $h(v)$ will be termed \emph{unbiased} if the resulting graph $G'$ has the following properties:
 \begin{itemize}
  \item Whenever $|d(u)| = |d(v)|$ for some $u,v \in V(G)$, the graphs induced on $h(u)$ and $h(v)$ are isomorphic.
  \item $\forall \: u,v \in V(G), \; \exists \: \gamma \in \textrm{ Aut}(G') \;\textrm{ such that }\; \gamma : h(u) \mapsto h(v)\;$ if and only if \newline $\exists \: \varphi \in \textrm{ Aut}(G)  \;\textrm{ such that }\;  \varphi : u \mapsto v,$
  \item For any $x,y \in V(G')$ such that $x \in h(u)$ and $y \in h(v)$ where $u \neq v$, we have $\{x,y\} \in E(G')$ only if $\{u,v\} \in E(G)$.
 \end{itemize}
\end{definition}

Only the first two properties are strictly necessary for the spirit of the term unbiased to hold, however the third property is included for ease of analysis. For instance, an alternative definition lacking the third requirement (but retaining the second) would allow gadgets which possess mutual connections if and only if the corresponding vertices of the initial graph are \emph{unconnected}.
We note that the third property holds for the sets of counterexamples proposed in both \cite{CFI} and \cite{Evdokimov99}.


\subsection{Properties of the CFI graph extension}
\label{subsec:CFI}

We will begin with the graph extension considered in \cite{CFI}, defined in Section \ref{sec:CFI}.

\begin{definition}
 Let the function $\Lambda : G \rightarrow \textrm{ CFI}(G)$ represent the extension of a graph obtained by replacing all vertices with their corresponding F\"{u}rer gadgets, connected as in Section \ref{sec:CFI}.
\end{definition}

To simplify some of the later analysis, we also introduce the notation:

\begin{definition}
 Given a vertex $x \in V(\textrm{CFI}(G))$, consider the function 
 \begin{align}
 \lambda^{-1} &: V(\textrm{CFI}(G)) \rightarrow V(G), \notag\\
 \lambda^{-1} &: x \mapsto v, \quad\;\; \forall x \in \textrm{CFI}(x),
 \end{align}
 which reverses the above process, mapping any vertex in the gadget CFI($v$) to the vertex $v \in V(G)$.
\end{definition}

The graph CFI($G$) has several important properties. Given a vertex $v \in V(G)$, the pair of vertices $a_{v,i},b_{v,i} \in \textrm{CFI}(v), i \in d(v)$ belong to the same orbit of Aut(CFI($G$)), and hence to the same colour class in $\textrm{WL}_k(\textrm{CFI}(G))$.

Similarly, all central vertices $m_S \in M(v)$ of a given gadget CFI($v$) also belong to the same orbit of Aut(CFI($G$)), and thus the same colour class of $\textrm{WL}_k(\textrm{CFI}(G))$.

With the exception of the case where $G$ is a cycle graph (this trivial case is assumed from this point to not occur), the following further properties regarding the relative colouring of the $A,B,M$ vertex sets also hold.

Since the $M(v)$ vertices are initially coloured differently to the $A(v)$ and $B(v)$ vertex sets, the 2-tuples $(a_{v,x},b_{v,x})$ and $(a_{v,x},b_{v,y})$, where $v,x,y \in V(G), x \neq y$ are assigned different colours by the $k$-dim WL method (for $k > 1$).

Another simple corollary of the definition of the CFI graph extension together with the above observations is that for all $u,v \in V(G)$, $u \neq v$, $$\textrm{Sort}[\textrm{WL}(\textrm{CFI}(u))] \;\textrm{ and }\; \textrm{Sort}[\textrm{WL}(\textrm{CFI}(v))]$$ are either equal or disjoint.

Before presenting a prove of Theorem \ref{thm:cfi(x)} we will focus on a simpler, `warm-up' case.

\begin{theorem}
 Given a graph $G$ with vertices $u,v \in V(G)$, $\textrm{WL}_1(u) \neq \textrm{WL}_1(v)$ if and only if $\;\textrm{Sort}[\textrm{WL}_1(\textrm{CFI}(u))] \neq \textrm{ Sort}[\textrm{WL}_1(\textrm{CFI}(v))]$
\end{theorem}

\begin{proof}
 Consider the vertices $u,v \in V(G)$ with associated gadgets $\textrm{CFI}(u), \textrm{ CFI}(v) \subset \textrm{ CFI}(G)$. Given $a \in \textrm{ CFI}(u), b \in \textrm{ CFI}(v)$, we have 
 \begin{equation}
   d(a) = \left\{
    \begin{array}{rl}
     |d(u)| & \textrm{if } a \in M(u),\\
     2^{|d(u)|-2}+1 & \textrm{otherwise}.
    \end{array} \right.
 \end{equation}
 Hence $\textrm{WL}^{0}_{1}(a) = \textrm{WL}^{0}_{1}(b)$ implies that $\textrm{WL}^{0}_{1}(u) = \textrm{WL}^{0}_{1}(v)$. Furthermore, if either $a \in M(u)$ and $b \in M(v)$, or $a \notin M(u)$ and $b \notin M(v)$ holds, then the converse is true, and we have $\textrm{WL}^{0}_{1}(a) = \textrm{WL}^{0}_{1}(b)$ if and only if $\textrm{WL}^{0}_{1}(u) = \textrm{WL}^{0}_{1}(v)$. 

 Conversely, assume that $\textrm{WL}^{i}_{1}(a) = \textrm{WL}^{i}_{1}(b)$ implies that $\textrm{WL}^{i}_{1}(u) = \textrm{WL}^{i}_{1}(v)$ for some $i \in \mathbb{Z}$. For any $x \in V(G)$,
 \begin{align}
  \textrm{WL}^{i+1}_{1}(x) = \langle \: \textrm{WL}^{i}_{1}(x), \textrm{ Sort}\{ \textrm{WL}^{i}_{1}(y) : y \in d(x)\}, \textrm{ Sort}\{\textrm{WL}^{i}_{1}(y) : y \in e(x)\} \:\rangle,
 \end{align}
 hence $\textrm{WL}^{i+1}_{1}(a) = \textrm{WL}^{i+1}_{1}(b)$ implies that $\textrm{ Sort}\{\textrm{WL}^{i}_{1}(y) : y \in d(a)\} = \textrm{ Sort}\{\textrm{WL}^{i}_{1}(y) : y \in d(b)\}$, and similarly for elements of $e(a)$ and $e(b)$. This in turn implies that $\textrm{WL}^{i+1}_{1}(u) = \textrm{WL}^{i+1}_{1}(v)$. Hence by induction we have
 \begin{align}
  \textrm{WL}_1(u) \neq \textrm{WL}_1(v) \textrm{ only if } \textrm{WL}_1(a) \neq \textrm{WL}_1(b).
 \end{align}
 Similarly, the converse follows if we restrict $a$ and $b$ such that either $a \in M(u)$ and $b \in M(v)$, or $a \notin M(u)$ and $b \notin M(v)$ holds, or if we consider the sorted set of colour classes associated with CFI($u$) and CFI($v$).
\end{proof}

A similar induction proof technique can be used to show that this result also holds for $k$-dim WL, for any $k$. A few requisite properties of the CFI graphs will first be established. Let $G$ be a graph, with $u,v,x,y \in V(G)$. Note that $\textrm{WL}_k(u,v) = \textrm{WL}_k(x,y)$ only if the number of paths of \emph{each} length connecting $u,v$ and $x,y$ respectively are equal \cite{Alzaga10}. Further, let the \emph{character} of a path $(x_1,\ldots,x_t)$ denote the ordered set of colour classes associated with each element of the path, $(\textrm{WL}_k(x_1), \ldots, \textrm{WL}_k(x_t))$. Then $\textrm{WL}_k(u,v) = \textrm{WL}_k(x,y)$ only if the number of paths of each length and of each character connecting $u,v$ and $x,y$ respectively are equal. Hence the following hold.

Let $G$ be a graph, with $u,v,w \in V(G)$, where $u \neq v, u \neq w$.
\begin{lemma}
\label{lem:WL1}
 Let $x_1,x_2 \in \textrm{CFI}(u), y \in \textrm{CFI}(v)$. Then for $k > 1$, $\textrm{WL}_k(x_1,x_2) \neq \textrm{WL}_k(x_1,y)$.
\end{lemma}

\begin{lemma}
\label{lem:WL2}
 Let $\{u,v\} \in E(G), \{u,w\} \notin E(G)$. Let $x \in \textrm{CFI}(u), y \in \textrm{CFI}(v), z \in \textrm{CFI}(w)$. Then $\textrm{WL}_k(x,y) \neq \textrm{WL}_k(x,z)$.
\end{lemma}

\begin{corollary}
 \label{cor:iso}
 Let $x,y,z \in \textrm{CFI}(G)$. Then $$\textrm{WL}_k(x,y) = \textrm{WL}_k(x,z) \;\;\textrm{ only if }\;\; \textrm{iso}(\lambda^{-1}(x),\lambda^{-1}(y)) = \textrm{iso}(\lambda^{-1}(x),\lambda^{-1}(z)).$$
\end{corollary}

\begin{proof}
 Firstly, recall that $\textrm{iso}(x_1 \ldots x_k) = \textrm{ iso}(y_1 \ldots y_k)$ if and only the relevant vertex and edge colourings match, and
 \begin{align}
  &x_i = x_j \textrm{ if and only if } y_i = y_j, \textrm{ and} \notag\\
  &(x_i,x_j) \in E(G) \textrm{ if and only if } (y_i,y_j) \in E(G).
 \end{align}
 Let $\lambda^{-1}(x) = u, \lambda^{-1}(y) = v, \lambda^{-1}(z) = w$. If the pairs $(u,v)$ and $(u,w)$ have different initial colours in $G$, then trivially the pairs of gadgets $(\textrm{CFI}(u),\textrm{CFI}(v))$ and $(\textrm{CFI}(u),\textrm{CFI}(w))$ have different initial colours in $\textrm{CFI}(G)$. The colour of $\textrm{iso}(u,v)$ further reflects which of the following holds:
 \begin{itemize}
  \item[(i)] $u = v$
  \item[(ii)] $\{u,v\} \in E(G)$
  \item[(iii)] $\{u,v\} \notin E(G)$, $u \neq v$.
 \end{itemize}
 In each case, by Lemmas \ref{lem:WL1} and \ref{lem:WL2} this information is also contained in the colouring of $\textrm{WL}_k(x,y)$.
\end{proof}

Hence a generalisation of the $k=1$ result to all $k$ can be derived.

\begin{theorem}
\label{thm:kdimCFI}
 Given a graph $G$, where $u,v \in V(G)$, 
 \begin{align*}
  \textrm{WL}_k(u) \neq \textrm{WL}_k(v) \quad \textrm{only if} \quad \textrm{Sort}[\textrm{WL}_k(\textrm{CFI}(u))] \neq \textrm{Sort}[\textrm{WL}_k(\textrm{CFI}(v))].
 \end{align*}

\end{theorem}

\begin{proof}
 In Chapter \ref{sec:WLprops}, for a given $t$-tuple $\textbf{z} \in V(G)^t$, $\;t < k$, we define $\textrm{WL}^{t}_k(\textbf{z})$ recursively by $$\textrm{WL}^{t}_k(\textbf{z}) = \langle \:\textrm{Sort}\{ \textrm{WL}^{t}_k((\textbf{z},i)) : i \in V(G)\} \:\rangle.$$
 Hence by Corollary \ref{cor:iso} above,
 \begin{align}
 \label{initial}
  \textrm{WL}^{0}_{k}(u) \neq \textrm{WL}^{0}_{k}(v) \textrm{ only if } \textrm{ Sort}[\textrm{WL}_{k}(\textrm{CFI}(u))] \neq \textrm{ Sort}[\textrm{WL}_{k}(\textrm{CFI}(v))].
 \end{align}
 Let $S_1,S_2 \subset V(G)$ be $k$-tuples of $G$, and $R_1,R_2 \subset V(\textrm{CFI}(G))$ be $k$-tuples of CFI($G$), such that $\lambda^{-1}(R_i) = S_i$, for $i \in \{1,2\}$.

 Assume that for some $m,n \in \mathbb{Z}$, $\textrm{WL}^{m}_k(S_1) \neq \textrm{WL}^{m}_{k}(S_2)$ implies that $\textrm{WL}^{n}_{k}(R_1) \neq \textrm{WL}^{n}_{k}(R_2)$, for all $S_1,S_2 \subset V(G)$, and all such $R_1,R_2$.

 Then consider a \emph{specific} set of $k$-tuples $S_1,S_2,R_1,R_2$ with the property
 \begin{align}
  &\textrm{WL}^{m}_{k}(S_1) = \textrm{WL}^{m}_{k}(S_2), \notag\\
  &\textrm{WL}^{m+1}_{k}(S_1) \neq \textrm{WL}^{m+1}_{k}(S_2).
 \end{align}
 In other words,
 $$ \textrm{Sort}\{\textrm{WL}^{m}_{k}(S^{'}_{1}(x) : x \in V(G)\} \neq \textrm{Sort}\{\textrm{WL}^{m}_{k}(S^{'}_{2}(x) : x \in V(G)\},$$
 where $S^{'}(x)$ is defined as in equation (\ref{eq:S'}) to be the set of `neighbouring' $k$-tuples to $S$, containing $x$.
 But by the assumption above, this implies
 $$ \textrm{Sort}\{\textrm{WL}^{n}_{k}(R^{'}_{1}(x) : x \in V(\textrm{CFI}(G))\} \neq \textrm{Sort}\{\textrm{WL}^{n}_{k}(R^{'}_{2}(x) : x \in V(\textrm{CFI}(G))\},$$
 hence $\textrm{WL}^{n+1}_{k}(R_1) \neq \textrm{WL}^{n+1}_{k}(R_2)$.

 Hence, as this assumption holds for $m=0$ it holds for all $m$ by induction.
\end{proof}

One final preliminary is needed before addressing Theorem \ref{thm:cfi(x)}.

\begin{lemma}
 \label{thm:k+1}
 If the recursive $k$-dim WL method refines the graph $G$ to its orbits at each step, then the recursive $(k+1)$-dim WL method refines the ordered 2-tuples of $G$ to their orbits at each step.
\end{lemma}

\begin{proof}
 This follows relatively directly from the definition of the $k$-dim WL method. Consider the pointwise stabiliser of $\textrm{WL}_k(G)$, in which a single vertex belonging to a particular colour class has been individualised. At each step of the recursive individualisation, the $k$-dim WL method again refines the resulting graph to its orbits. Denote the graph resulting from applying the $k$-dim WL method, then individualising some vertex (belonging to an orbit of size greater than one) $i$ times by $G^{(i)}_{k}(T)$, where $T = (t_1, \ldots, t_i)$, $t_j$ denoting the vertex individualised at step $j$ in the recursive method.

 Given two ordered pairs $(u,v), (x,y) \in V \times V$, assume that
 \begin{align}
 \label{eq:pairs}
  (\textrm{WL}_k(u),\textrm{WL}_k(v)) &= (\textrm{WL}_k(x),\textrm{WL}_k(y)), \textrm{ and} \notag\\
  \textrm{WL}_k(v) |_{G^{(1)}_{k}(u)} &= \textrm{WL}_k(y) |_{G^{(1)}_{k}(x)}.
 \end{align}
 Denote by $S_k(a,b)$ the set of $(k)$-tuples containing $a$ and $b$,  $S_k(a) = {R \in V(G)^{k} : a,b \in R}$. Then since $G^{(1)}_{k}(a)$ is refined to its orbits by the $k$-dim WL method, equations \ref{eq:pairs} above imply that
 \begin{align}
  (\textrm{WL}_{k+1}(u),\textrm{WL}_{k+1}(v)) &= (\textrm{WL}_{k+1}(x),\textrm{WL}_{k+1}(y)), \textrm{ and} \notag\\
  \textrm{WL}_{k+1}(S_{k+1}(u,v)) &= \textrm{WL}_{k+1}(S_{k+1}(x,y),
 \end{align}
 and hence that $\textrm{WL}_{k+1}((u,v)) = \textrm{WL}_{k+1}((x,y))$.
 Hence the colour of ordered 2-tuples in $\textrm{WL}_{k+1}$ is equal only if they belong to the same orbit of $V \times V$ in $G$.
\end{proof}

Combining the results of Theorem \ref{thm:kdimCFI}, Lemma \ref{thm:k+1} and the observations regarding the properties of $\textrm{WL}_k(\textrm{CFI}(G))$ at the start of this subsection, we can now return to Theorem \ref{thm:cfi(x)} introduced at the beginning of Chapter \ref{sec:ext}.

\begin{theorem}
 If the recursive $k$-dim WL method succeeds in characterising the graph $G$, then the recursive $(k+1)$-dim WL method succeeds in characterising the graph CFI($G$).
\end{theorem}

\begin{proof}
 Assume $\exists \, u,v \in V(G)$, such that $\textrm{WL}_k(u) \neq \textrm{WL}_k(v)$. Then by Theorem \ref{thm:kdimCFI} for any $r,s \in V(\textrm{CFI}(G)), \lambda^{-1}(r) = u, \lambda^{-1}(s) = v$, it follows that $\textrm{WL}_k(r) \neq \textrm{WL}_k(s)$. Furthermore the graph induced on any given gadget of CFI($G$) is itself refined to its orbits (in that the $A$/$B$ subsets are separated from the $M$ subset in non-trivial cases). For a given $v \in V(G)$, the orbits of the set of central vertices $M(v) \subset \textrm{ CFI}(v)$ depend only on the orbit of the vertex  $v$ in Aut($G$). Hence the central $M$ vertices of CFI($G$) are refined to their orbits by the $k$-dim WL method, and hence by the $(k+1)$-dim WL method. However the $A$ and $B$ vertices of each gadget CFI($v$) correspond to di-edges of $G$, in that the orbit of a particular vertex $a_{u,v}$ depends on the orbit of the ordered 2-tuple $(u,v)$ in Aut($G^2$). Explicitly, $\exists \, \phi \in \textrm{ Aut(CFI}(G)), \phi : a_{u,v} \mapsto a_{x,y}$ if and only if $\exists \, \gamma \in \textrm{ Aut}(G^2), \gamma : (u,v) \mapsto (x,y)$.

 Now as the set of colour classes from the WL method is a graph invariant, and so cannot refine further than the orbit partition, it follows from the previous argument that the $n$-dim. WL method refines CFI($G$) to its orbits if it refines the ordered 2-tuples of $G$ to their orbits. Hence the result follows from Lemma \ref{thm:k+1}.
\end{proof}

\begin{corollary}
 If $G$ can be characterised by the recursive $k$-dim WL method, then the graph $\Lambda^{i} (G)$, in which $\Lambda$ is applied $i$ times to $G$, can be characterised by the $(k+1)$-dim WL method.
\end{corollary}

\begin{proof}
 Follows from the proof of \ref{thm:cfi(x)}. In particular, note that applying $\Lambda$ $i$ times to $G$ still results in a graph in which vertices represent at worst di-edges of $G$, in that the automorphism group of $\Lambda^{i} (G)$ does not involve automorphisms of $t$-tuples of $G$ for $t > 2$.
\end{proof}

Note that whereas the recursive $k$-dim WL method is sufficient to characterise $G$, the $(k+1)$-dim WL method is required to characterise the extension CFI($G$). This is a direct result of there being vertices in CFI($G$) that directly represent di-edges, or ordered \emph{pairs} of vertices, in $G$. Similarly, if a graph extension was constructed that contained vertices representing 3-tuples of $G$, the $(k+2)$-dim WL method would be required in a proof proceeding as above. There are some fairly contrived possibilities for getting around this requirement, for instance by altering the recursive WL method, restricting the vertex individualised at each step to belong to the set $M_u$ for some $u \in V(G)$. Since these central vertices `encode' only a single vertex in $G$, unlike the $A$ and $B$ sets they will necessarily be refined to their relative orbits by the $k$-dim WL method. Alternatively, the recursive WL method (when acting on a graph of the type CFI($G$)) could be restricted to individualising an entire gadget at each step. These alternative method require foreknowledge of the graph type of interest however, and as such are of less interest to a discussion on possible general graph isomorphism algorithms.

Requiring an extension to the $(k+1)$-dim method may seem prohibitive from an efficiency viewpoint, in that extensions similar to those in \cite{CFI} and \cite{Evdokimov99}, in which vertices are present whose orbit depends on the orbit of some $t$-tuple in the original graph, for $t > 2$, might easily be produced. A method involving such $t$-tuples in an unbiased way, in the sense of definition \ref{def:unbiased}, might be expected to develop alternative weaknesses with growing $t$ however.
For the moment such potential extensions are beyond the scope of this work, however they do represent a potentially promising direction in which to look for $k$-equivalent graphs for which the arguments of this chapter do not apply.

\subsection{$K(G)$ Counterexamples}

Here we consider the $k$-equivalent pairs $K(G)$ and $K'(G)$, given a cubic graph $G$. The analysis of these counterexamples is greatly simplified due to the following properties:

Let $\mathcal{C} = (V,\mathcal{R})$ be the coherent configuration associated with $K(G)$, and $\textrm{Aut}(\mathcal{C})$ be the automorphism group of $\mathcal{C}$, where $\textrm{Aut}(\mathcal{C}_{i,j})$ is the automorphism group of the induced coherent configuration on $V_i \times V_j$. Then, from the results of \cite{Evdokimov99} (Lemma 5.3),
\begin{equation}
 |\textrm{Aut}(\mathcal{C}_{i,j})| = \left\{
 \begin{array}{rl}
  4 & \textrm{if } i=j\\
  8 & \textrm{if } (i,j) \in E(G)\\
  16 & \textrm{otherwise}.
 \end{array} \right.
\end{equation}

In particular, note that no fibres span more than one $V_i$, and the orbits of $V$ are simply the sets $V_i, \;\forall i \in V(G)$. Further, let $\mathcal{C}^*$ be the coherent configuration resulting from individualising a single point of $\mathcal{C}$, where $K(G)^*$ is the associated graph.

\begin{corollary}
 $|\textrm{Aut}(\mathcal{C}^*)| = 1$; no non-trivial automorphisms exist for $\mathcal{C}^*$.
\end{corollary}

Consider the following adaption of the $1$-dim WL method, accepting an edge-coloured graph as input. Denote the colour of the edge $(u,v) \in E(K(G))$ by $C(u,v)$, and let
\begin{align}
 \textrm{WL}^{t}_1(u) = \langle \,\textrm{WL}^{t-1}_1(u), \textrm{ Sort}\{ (\textrm{WL}^{t-1}_1(v), &C(u,v)) : v \in d(u)\},  \notag\\
 &\textrm{ Sort}\{(\textrm{WL}^{t-1}_1(v) : v \in e(u)\} \:\rangle.
\end{align}

Theorem \ref{thm:kleinWL} follows immediately.

\begin{reptheorem}{thm:kleinWL}
 \emph{The recursive $1$-dim WL method succeeds in characterising the graph $K(G)$, associated with a Klein scheme of $G$.}
\end{reptheorem}

\begin{proof}
 Let $u,v \in V(K(G))$, where $u \in V_i, v \in V_j$ for some $i \neq j$. If follows that $u$ and $v$ have incident edge-colour sets that do not coincide. Then $\textrm{WL}^{1}_1(u) \neq \textrm{WL}^{1}_1(v)$. Hence the $1$-dim WL method described above refines $K(G)$ to its orbits.

 Let $u \in V_i$ be the vertex of $K(G)$ individualised in $K(G)*$. The remainder of $V_i$ are connected to $u$ each via edges of a different colour, hence for any $x,y \in V_i$, $\textrm{WL}^{1}_1$ assigns different colours to $x,y$. Similarly, consider $V_j$ such that $i$ and $j$ are at distance $n$ in $G$. Then $\textrm{WL}^{n}_1$ assigns different colours to each vertex of $V_j$. Hence $K(G)*$ is refined to its orbits (the discrete partition) by $\textrm{WL}_1$.
\end{proof}

The ease of proving this result compared to the corresponding result regarding the CFI counterexamples stems from each `gadget' in $K(G)$ being essentially assigned a unique colour (explicitly in \cite{Evdokimov99}, implicitly here via the the unique colouring of each di-edge of $G$). This distinction between di-edge colouring could be removed, with the $\mathcal{R}_{i,j}$ basis relations of Chapter \ref{sec:schemes} merged, such that the sets of basis relations
\begin{align}
 \mathcal{S}_x &= \hspace{-0.4cm} \bigcup_{\{i,j : (i,j) \in E(G)\}} \hspace{-0.6cm} R_x, \quad\quad x \in \{1,2\}, \; \textrm{ and}\\
 \mathcal{T} &= \hspace{-0.4cm} \bigcup_{\{i,j : (i,j) \notin E(G)\}} \hspace{-0.6cm} \mathcal{R}_{i,j}\\
\end{align}
are each merged into a single basis relation, forming three distinct subsets of $\mathcal{C}$ \footnote{Explicitly, leaving the three sets of relations; $\mathcal{S}_1$, $\mathcal{S}_2$ and $\mathcal{T}$}, and with the basis relations of all individual $V_i$ being merged into three basis relations similarly.

We will state without proving the following proposition (which follows from the results of \cite{Evdokimov99}).
\begin{proposition}
 Merging the basis relations as detailed above preserves the properties of $k$-similarity and non-isomorphism between $\mathcal{C}$ and $\mathcal{C}'$.
\end{proposition}

Indeed, the automorphism group of the resulting coherent configuration has orbits with the following properties. Vertices $u,v$ are in the same orbit if they belong to a single $V_i$. The vertices of $V_i$ and $V_j$ belong to the same orbit if and only if $i,j$ are in the same orbit of $G$.

Hence the question of whether $\textrm{WL}_k$ refines $K(G)$ to its orbits reduces to a similar problem as that of the previous section. A similar proof can be constructed, with one important distinction. Since the vertices of each $V_i$ correspond to a single vertex of $G$, rather than a $2$-tuple of $G$ as for CFI($G$), the following result is obtained.

\begin{proposition}
 If the $k$-dim WL method succeeds in refining $G$ to its orbits, then it succeeds in characterising $K(G)$ to its orbits. 
\end{proposition}

The proof is along the same lines as that of the previous section, although considerably simpler due to the above observation.

\subsection{General graph extensions}

Recall the generalised graph extension $G \rightarrow G'$ defined in \ref{def:unbiased} (relative to an `unbiased' gadget) at the beginning of this chapter. 

\begin{repdefinition}{def:unbiased}
 The extension of a graph $G$ by replacing each vertex $v \in V(G)$ by some type of gadget $h(v)$ will be termed \emph{unbiased} if the resulting graph $G'$ has the following properties:
 \begin{itemize}
  \item Whenever $|d(u)| = |d(v)|$ for some $u,v \in G$, the graphs induced on $h(u)$ and $h(v)$ are isomorphic.
  \item $\forall u,v \in V(G), \exists \, \gamma \in \textrm{ Aut}(G'), \gamma : h(u) \mapsto h(v) \textrm{ iff } \exists \, \varphi \in \textrm{ Aut}(G), \varphi : u \mapsto v,$
  \item For any $x,y \in V(G')$ such that $x \in h(u)$ and $y \in h(v)$, we have $\{x,y\} \in E(G')$ only if $\{u,v\} \in E(G)$.
 \end{itemize}
\end{repdefinition}

Further to this definition of an unbiased gadget, the graph extensions considered will be assumed to have the following property relative to the $k$-dim WL method.

We assume that pairs of gadgets corresponding to neighbouring vertices of $G$ are distinguished from those corresponding to non-edges of $G$.
Note that the graph extensions defined in \cite{CFI} and \cite{Evdokimov99} satisfy this property.

This assumption implies that the colour of a $k$-tuple in $G'$ depends on the isomorphism class of the corresponding $k$-tuple in $G$. Hence by induction, as in Section \ref{subsec:CFI}, the colour of a $k$-tuple in $G'$ also depends on the \emph{colour class} (resulting from the $k$-dim WL method) of the corresponding $k$-tuple in $G$. Hence this implies that the $k$-dim WL method partitions the \emph{gadgets} of $G'$ into the same relative colour classes as it partitions the \emph{vertices} of $G$.

\setcounter{equation}{0}
\section{Orbit Case}
\label{sec:orbit}

Up to this point we have been interpreting the results of \cite{CFI, Evdokimov99} as demonstrating that the $k$-equivalent CFI pairs cannot be \emph{directly} distinguished, in that any method of producing a canonical certificate from each graph using the colour sets resulting from the $k$-dim WL method will yield identical certificates. However we will argue that this does not imply that the WL method (or some variant of this method) cannot be used indirectly to solve GI. In particular, recall that among the problems polynomial-time equivalent to the GI problem is that of determining the orbits of the automorphism group of a given graph. A method that can partition the vertex set of any graph down to its orbits can trivially solve GI by simply refining the graph to its orbits, stabilising some vertex from a given orbit, then accepting the resulting graph as input and repeating until the discrete partition is reached. At this point an explicit non-trivial automorphism of the graph will have been found (if any such exist), and the method can be repeated to find a set of generators for the automorphism group. Alternatively, a method that can partition the vertex set of any graph down to its orbits can simply be directly applied to the union of two graphs to determine if they belong to the same isomorphism class.

Hence the focus of the following chapters will be purely on the equitable partitions resulting from the WL method, rather than the related graph invariant.

\subsection{Counterexamples (orbit case)}
\label{sec:orbit_counter}

If the goal of the WL method is instead considered to be refining the vertex set of a graph down to its orbits, the set of known counterexamples differs, since graphs that cannot be directly distinguished by comparing certificates could still be indirectly distinguished by a recursive WL method, provided that at each step of the recursion the orbits could be successfully found.

Such a recursive $k$-dim WL method would proceed as follows. Apply the $k$-dim WL method to the graph until the equitable partition is reached, at which point this partition is assumed to be an ordered set of orbits of the graph. Without loss of generality, any vertex from the lexicographically smallest (for example) orbit is then stabilised, and this process is repeated until the discrete partition is reached. If at each step the equitable partition corresponds to the orbit set, a canonical ordering of the vertex set is obtained, characterising the graph. Alternatively, by choosing different sets of representative vertices to stabilise at each step, generators for the automorphism group can be efficiently obtained.

The success of this procedure is of course dependent on the ability to discover the orbits at each step. However, we will briefly note that the this method may possibly succeed for graphs where a direct application of the WL method fails, and additionally that both success and failure of this method occur in polynomial time (together with the knowledge of which has occurred\footnote{A simple polynomial time extension to the method can be used to determine generators for the complete automorphism group, which in turn can be efficiently verified.}).

Now the results of \cite{CFI, Evdokimov99} demonstrate that the $k$-equivalent pairs described cannot be distinguished directly, and hence that, for instance, the union of such a pair cannot itself be partitioned down to its orbits using the $k$-dim WL method.
In this trivial case where the graph under consideration is simply the union of two $k$-equivalent graphs, say CFI($G$) and CFI$'$($G$) for some connected graph $G$, this inability to determine the orbits of the combined graph $H = \textrm{CFI}(G) \cup \textrm{CFI}'(G)$ directly can be easily circumvented, under certain assumptions. Provided the original graph $G$ can be refined to its orbits using the $k$-dim WL method, and further that this refinement can be recursively applied (with pointwise stabilisation at each step) to completely characterise the graph (calculating its automorphism group), then in Chapter \ref{sec:ext} we show that CFI($G$) and CFI$'$($G$) can also be separately characterised by recursive application of the $(k+1)$-dim WL method. Then since $H$ can be trivially decomposed into its $k$-equivalent sub-constituents (by separating the disconnected components of either $H$ or $\overline{H}$, depending on the definition of the `union' of graphs), we can apply the recursive WL method to each of the two $k$-equivalent subgraphs individually, determine that these are non-isomorphic graphs, and hence characterise the composite graph $H$.

The above example applies a method for decomposing a graph for which the WL method has been proven to fail into its $k$-equivalent subgraphs.
Of course, such a method cannot be so easily applied in general. Before considering how a generalised decomposition method might proceed, we will consider precisely what types of graphs have been found for which the WL method has been proven (in the work of \cite{CFI, Evdokimov99}) to \emph{not} determine the orbits. The union of two or more $k$-equivalent graphs as considered above is one such graph, albeit a trivial case.

Consider a graph $H$ that can be directly deduced from the results of \cite{CFI, Evdokimov99} to not be partitioned down to its orbits by the $k$-dim WL method. Then trivially, $H$ must contain at least two mutually $k$-equivalent subgraphs, $S_1$ and $S_2$. Any differences in the way these subgraphs are connected to the remainder of the graph, relative to the colour classes resulting from the $k$-dim WL method, may distinguish them (or at least have not been proven not to do so), hence assume no such differences exist. 
Furthermore, consider a single orbit of a given $k$-equivalent graph $S$. Any difference in the connections of the elements of this orbit to the rest of the graph may yield enough information regarding the internal structure of $S$ such that its property of $k$-equivalence is destroyed. Hence such differences will also be assumed to not exist\footnote{Note that in this work we are only concerned with graphs for which the recursive $k$-dim WL method has been proven to fail; we wish to know exactly what counterexamples can be directly constructed from the results of \cite{CFI, Evdokimov99}, hence such differences cannot be allowed without a further extension to this work.}.

If we consider the orbit partition of $H$, in which each cell of the partition is a distinct orbit of Aut($H$), then such a subgraph $S$, essentially a generalisation of a module of a graph, will be termed to be connected \emph{cell-wise symmetrically} (CWS) to the remainder of the graph (relative to the orbit partition in this case), defined as follows.

\begin{definition}
\label{def:CWS}
 Consider an ordered partition $\pi(G) = (V_1, V_2, \ldots, V_r)$ of the graph $G$ into cells (or colour classes), and define $\theta : V(G) \rightarrow \pi(G)$ to determine the cell of a given vertex $v \in V(G)$. Then a subgraph $S \subset G$ is connected \emph{cell-wise symmetrically} (or CWS) within $G$, alternatively termed a \emph{cell-wise symmetric} (or CWS) subgraph of $G$, relative to $\pi(G)$, if:
 \begin{align}
  &\forall\; u,v \in V(S) \textrm{ such that } \theta(u) = \theta(v), \,\textrm{ we have }\, d(u)|_{(G\backslash S)} = d(v)|_{(G\backslash S)}.
 \end{align}
 In other words, any elements of $S$ in the same cell of $\pi(G)$ have identical neighbour sets outside $S$ (in $G\backslash S$).
\end{definition}

\begin{definition}
\label{def:prime}
 A subgraph $S \subseteq G$ will be termed \emph{prime} if it has no proper, non-trivial CWS subgraphs, and is itself a non-trivial CWS subgraphs of $G$. This is defined implicitly with respect to the $k$-dim WL method.
\end{definition}

\begin{example}
 Let $G$ be the $n$-cube, in which the vertices are associated with the related $2^n$ points in $\mathbb{Z}^{n}_2$. Associate with $G$ the partitioning $\pi(G) = (\{x\},V(G) \backslash \{x\})$, where $x \in V(G)$ is the vertex with associated bit-string $(0\ldots0)$. This would be for instance the orbit set resulting from individualising the vertex $x$. Then the subset $S_c \subset V(G)$, $S_c = \{ v \in V(G) : \textrm{dist}(v,x) = c\}$, corresponding to the set of points with fixed Hamming weight $c$, is a CWS partition of $G$ relative to $\pi(G)$.
\end{example}

For the remainder of this paper, we will consider CWS subgraphs to be defined relative to either the orbit partition or the colour classes resulting from the $k$-dim WL method. In the former case such subsets will be referred to as \emph{orbit-wise symmetric (OWS)} subsets, and in the latter case simply as simply CWS subsets, with the `relative to the colour classes resulting from the $k$-dim WL method' specifier dropped for the purposes of brevity.

The notion of CWS subgraphs of a graph can be extended to \emph{relative} connections between mutually $k$-equivalent subgraphs of a graph. In particular, the properties of known counterexample graphs discussed above refer to differences in the relative connections between non-isomorphic $k$-equivalent subgraphs and the remainder of the graph. Before formalising this concept into a definition of \emph{mutual CWS subsets}, it will be instructive to consider a simpler case.

In particular, let $R$ and $S$ be isomorphic, vertex-disjoint, mutually $k$-equivalent subgraphs of a graph $G$. As $R$ and $S$ are mutually $k$-equivalent, $\exists \, \theta : V(R) \rightarrow V(S)$ such that
\begin{align}
\label{eq:equiv_map}
 \textrm{WL}_k(u) \mid_R = \textrm{WL}_k(\theta(u)) \mid_S, \quad \forall \; u \in V(R).
\end{align}
We will assume that $R$ and $S$ are connected CWS within $G$, relative to the WL$_k$ colour classes corresponding to their respective induced graphs, such that for all $u,v \in V(R)$ such that $\textrm{WL}_k(u) \mid_R = \textrm{WL}_k(v) \mid_R$, we have
\begin{align}
 \textrm{WL}_k(u) \mid_G = \textrm{WL}_k(v) \mid_G,
\end{align}
and similarly for $S$.
Further, $R$ and $S$ will have the property that their cell-wise connections to $G \backslash R$ and $G \backslash S$ respectively are equivalent, in the sense that
\begin{align}
 \textrm{WL}_k(d(u)) \mid_G = \textrm{WL}_k(d(\theta(u))) \mid_G, \quad \forall \; u \in V(R),
\end{align}
which in turn implies that
\begin{align}
\label{eq:mutualCWS1}
 \textrm{WL}_k(u) \mid_G = \textrm{WL}_k(\theta(u)) \mid_G, \quad \forall \; u \in V(R).
\end{align}
Note that if (\ref{eq:mutualCWS1}) holds for one such mapping $\theta$ defined as in (\ref{eq:equiv_map}), it holds for all such mappings. In other words, in the terminology of Chapter \ref{sec:schemes}, there exists a weak $k$-automorphism of $G$ that maps $R$ to $S$, in that
\begin{align}
 \textrm{WL}_k(R) \mid_G = \textrm{WL}_k(S) \mid_G.
\end{align}
Finally, we will assume that $R$ and $S$ are the only mutually $m$-equivalent graphs, $m \ge k$, for which the above holds.

We will argue that in this situation, either there exists an automorphism of $G$ that maps $R$ to $S$, or $G$ represents a novel type of $k$-equivalent graph, the existence of which is currently unknown.

Firstly, note that the above properties are trivially consistent with the case where an automorphism of $G$ maps $R$ to $S$, and that in either case $G$ must be $k$-equivalent, in that replacing $S$ with a non-isomorphic mutually $k$-equivalent copy of $S$\footnote{Which exists by the assumption that $S$ is $k$-equivalent} while retaining the mapping $\theta$ results in a graph which is mutually $k$-equivalent to $G$.

In the situation where no such automorphism exists, we note that the graph $G'$ obtained by replacing $R$ and $S$ each by a single vertex\footnote{This process is explicitly defined in Definition \ref{def:contract}} ($v_R and v_S$ respectively), retaining the representative CWS connections outside $R$ and $S$ respectively, is also $k$-equivalent, in that
\begin{align*}
 \textrm{WL}_k(v_R) \mid_{G'} = \textrm{WL}_k(v_S) \mid_{G'},
\end{align*}
however no automorphism of $G'$ maps $v_R$ to $v_S$. Hence we are essentially shifting the $k$-equivalence property of $G$ outside $R$ and $S$.

Now the arguments directly prior to Definition \ref{def:CWS} essentially state that known classes of graphs which are not partitioned to their orbits by the $k$-dim WL method must contain a pair of non-isomorphic graphs with the properties ascribed to $R$ and $S$ above. However since we have assumed that no further such mutually $m$-equivalent graphs exist in $G$, for $m \ge k$, this does not occur, and hence the graph $G$ represents a novel type of $k$-equivalence. 

The above discussion serves to inform the following definition of \emph{mutually CWS} subsets of a graph.
Let $S=\{A_1,A_2,\ldots,A_i\}$ be a set of vertex-disjoint, CWS subgraphs of a graph $G$, such that the elements of $S$ are all pairwise $k$-similar. The graphs $A_x$, where $x \in [i]$, are said to be \emph{mutually CWS} if they have $k$-equivalent connections in $G$ in the following sense.

\begin{definition}
\label{def:mutual_k_equiv}
 Consider the graph $G'$ in which all pairs of non-isomorphic mutually $m$-equivalent subgraphs of $G$ are replaced by \emph{isomorphic} mutually $m$-equivalent graphs of the same $m$-equivalence class, for all $m \ge k$. In the case where the subgraphs $A_x$ are not $k$-equivalent, $S$ is unchanged. It follows that $G' \sim_k G$, however all mutually $m$-equivalent subgraphs of $G'$ now belong to a single isomorphism class. Then elements of $S$ are said to be \emph{mutually CWS} (or alternatively said to be connected CWS to each other) within $G$, if for all $A_x, A_y \in S$, there exists a $\phi \in \textrm{Aut}(G')$ such that $\phi(A_x) = A_y$ (and hence $S$ is in this sense vertex-transitive).
\end{definition}

Note that the definition of mutually CWS subgraphs corresponds to a specific value of $k$, which will be clear from the context where not explicitly stated.

\begin{corollary}
 If $R$ and $S$ are mutually CWS subgraphs of $G$ relative to the $k$-dim WL method, then $R \sim_k S$.
\end{corollary}

\begin{example}
 Let $S=\{A_1,A_2,\ldots,A_i\}$ be a set of vertex-disjoint, CWS, mutually $k$-equivalent subgraphs of a graph $G$. Then for all $A_x, A_y \in S$, $\exists \; \theta_{x,y} : V(A_x) \rightarrow V(A_y)$, such that $\forall u \in V(A_x)$, 
 \begin{align}
  \textrm{WL}_k(u) \mid_{A_x} = \textrm{WL}_k(\theta(u)) \mid_{A_y}.
 \end{align}
 In particular, for all $x,y \in [i]$ and $u \in V(A_x)$,
 \begin{align*}
  \textrm{Sort}[\textrm{WL}_k(d(u)) \mid_{(G \backslash S)} ] &= \textrm{Sort}[\textrm{WL}_k(d(\theta_{x,y}(u))) \mid_{(G \backslash S)} ], \textrm{ and}\\
  \exists \, A_z \in S \textrm{ s.t. } \textrm{Sort}[\textrm{WL}_k(d(u)) \mid_{A_y} ] &= \textrm{Sort}[\textrm{WL}_k(d(\theta_{x,z}(u))) \mid_{A_x} ].
 \end{align*}
\end{example}

\begin{example}
 A set $S=\{A_1,\ldots,A_i\}$ of CWS subgraphs of $G$ are trivially mutually CWS if the following hold
 \begin{itemize}
  \item[(i)] The $A_x \in S$ are all pairwise mutually $k$-equivalent.
  \item[(ii)] For all $x,y \in [i]$ and $u \in V_x$, $v \in V_y$ such that $u \sim_k v$, we have
  \begin{align}
   d(u) \cap (G \backslash V_x)  = d(v) \cap (G \backslash V_y).
  \end{align}
 \end{itemize}
\end{example}

Note that a graph $G$ containing a CWS $k$-equivalent subgraph $S$ is not necessarily itself $k$-equivalent, in that although the connections between $S$ and $G$ are CWS, they are not necessarily OWS. However several trivial cases, such as a CWS $k$-equivalent subgraph $S$ in which all vertices have identical neighbour sets outside of $S$, can be constructed in which the resulting graph $G$ \emph{is} provably also $k$-equivalent. Hence in addressing the known $k$-equivalent graphs, we will make the following assumption.

\begin{assumption}
 Any graph $G$ that contains a CWS $k$-equivalent subgraph is itself $k$-equivalent.
\end{assumption}

Whilst graphs exist for which this does not hold (in fact several were readily found in the course of this work), this assumption makes the following task of addressing all known counterexample graphs more difficult (in the sense that additional graphs must be considered), and so can be made without weakening the end results.

Following the preceding set of definitions, and the properties of known counterexample graphs discussed above, we now have the notation required to define the following family of graphs.

\begin{definition}
\label{def:M_k}
 The family $M_k$ of graphs is defined as \emph{including} those graphs for which the $k$-dim WL method has been shown to fail to partition the vertex set down to its orbits. and \emph{not} including graphs for which the $(k+1)$-dim WL method has been shown to fail in this sense. All graphs $G \in M_k$ have the following properties:
 \begin{enumerate}
  \item $G$ and $\overline{G}$ are connected.
  \item $G$ contains two non-isomorphic, mutually CWS subgraphs $S_1,S_2$, each connected CWS to $G$ relative to the colour classes of $\textrm{WL}_k(S_i) \mid_{S_i}$.
  \item $G$ contains no pair of non-isomorphic, mutually CWS, $m$-equivalent subgraphs, where $m > k$.
 \end{enumerate}
\end{definition}

Note that if (2) does not hold, and such a mutually CWS pair of subgraphs does not exist, then the $k$-dim WL method will not have been proven to fail to provide the orbits of $G$. Similarly if (3) does not hold, $G \in M_{k+1}$, hence $G \notin M_{k}$. 

Note that the set of graphs $M_k$ is not intended as a complete set of graphs that the $k$-dim WL method fails to refine down to orbits, but instead as including the set of graphs for which this has been previously proven to occur. For instance, it may be possible to construct a graph not containing mutually CWS $k$-equivalent graphs that still cannot be successfully characterised by the recursive $k$-dim WL method.
However such graphs have not been proven to exist, and the primary object of this paper is simply to discuss the possibility that a variant of the WL method might be used to solve GI, not to prove that it actually can.

Note that the set $M_k$ contains those graphs for which the $k$-dim WL method has been proven to fail, in the sense that it cannot partition the vertex set of such a graph down to its orbits. In Chapter \ref{sec:decomposition}, we will detail an algorithm employing the recursive WL method that can characterise these graphs, under certain assumptions, by first applying a decomposition method that isolates the relevant $k$-equivalent subgraphs. These subgraphs are then individually characterised using the standard recursive WL method, at which point the non-isomorphic $k$-equivalent subgraphs are distinguished, and the original graph can be characterised.

\setcounter{equation}{0}
\section{Properties of the WL method}
\label{sec:WLprops}

Before continuing our discussion regarding counterexample graphs, we will first take a brief interlude to establish some basic properties of the $k$-dim WL method. These properties will prove useful in constructing some of the proofs of Chapter \ref{sec:decomposition}. In particular, the relation between the colour classes of $\textrm{WL}_k(G)$ and the CWS subgraphs of $G$ will be explored. Firstly, the CWS closure $cl(S)$ of a subgraph $S \subset G$ is defined as the smallest CWS subset of $G$ containing $S$.

Consider the following method for constructing the CWS closure $cl(\{u,v\})$ of a pair of vertices $u,v \in V(G)$ belonging to the colour class $c$, such that $u \hspace{-0.06cm} \sim_k \hspace{-0.085cm} v$ in $G$, where $\textrm{WL}_k(u) = c$.
\begin{itemize}
 \item[(i)] Begin with $S^0 = \{u,v\}$
 \item[(ii)] For each $x,y \in S^i$ such that $\textrm{WL}_k(x) = \textrm{WL}_k(y)$, find the set \newline $R(x,y) := (d(x) \cap e(y)) \cup (e(x) \cap d(y))$.
 \item[(iii)] Set $S^{i+1} = (\bigcup_{x,y}R(x,y)) \cup S^i$.
 \item[(iv)] When $S^t$ is equitable, such that $S^t = S^{t+1}$, set $cl(\{u,v\}) = S^t$.
\end{itemize}

\begin{lemma}
\label{thm:onewayclosure}
 Let $\textrm{WL}_k(G)$ contain colour classes $c_1$ and $c_2$. If there exists a vertex $v \in [c_2]$ such that $v \in cl([c_1])$, then $[c_2] \subset cl([c_1])$.
\end{lemma}

\begin{proof}
Let $S^0 = \{x_1,x_2\}$, for $x_1,x_2 \in [c_1]$. If there exists a $u \in [c_3]$ for some colour class $c_3 \in \textrm{WL}_k(G)$, such that $u \in S^1$, then either
\begin{align*}
 &u \in d(x_1), u \notin d(x_2), \textrm{ or}\\
 &u \notin d(x_1), u \in d(x_2).
\end{align*}
Since $c_3$ is a colour class of $\textrm{WL}_k(G)$, then for all such $v \in [c_3]$,
$$ \exists \; x_i,x_j \in [c_1] \textrm{ s.t. } v \in d(x_i), \; v \notin d(x_j),$$
hence $[c_3] \subset cl([c_2])$.

Similarly, if there exists a $v \in [c_2]$ such that $v \in S^{i}$, where $S^{i-1} \cap [c_2] = \emptyset$, then one (or more) of the following hold:
\begin{itemize}
 \item $\exists \; x_i,x_j \in [c_1], \;\{x_i,x_j\} \subset S^{i-1}, \textrm{ s.t. } v \in d(x_i), \; v \notin d(x_j)$.
 \item $\exists \; c_3 \in \textrm{WL}_k(G), \; y_1,y_2 \in [c_3], \; \{y_i,y_j\} \subset S^{i-1}, \textrm{ s.t. } v \in d(y_i), \; v \notin d(y_j)$.
\end{itemize}
In the former case, $[c_2] \subset cl([c_1])$ as above. In the latter case, $[c_2] \subset cl([c_3])$, where $[c_3] \subset cl([c_1])$ in turn, hence $[c_2] \subset cl([c_1])$.
\end{proof}

\begin{corollary}[`No One-Way Closure']
 $[c_2] \subset cl([c_1])$ if and only if $[c_1] \subset cl([c_2])$.
\end{corollary}

\begin{proof}
 Follows immediately from the proof of lemma \ref{thm:onewayclosure}.
\end{proof}

This `no one-way closure' result only applies when considering the closure of \emph{entire} colour classes. If instead considering subsets $S_1 \subset c_1$, $S_2 \subset c_2$, then trivially we can have $S_1 \subset cl([S_2])$ and $S_2 \nsubseteq cl([S_1])$.

\vspace{0.5cm}
\noindent\textbf{\emph{A note on notation}}: In the remainder of this chapter we consider the colour class assigned to ordered $t$-tuples by the $k$-dim WL method, for varying $t$. In denoting the colour class of an some ordered $t$-tuple $(x_1,x_2,\ldots,x_t)$, we will use the notation $\textrm{WL}_k(x_1,x_2,\ldots,x_t)$ to refer to the more explicit $\textrm{WL}_k((x_1,x_2,\ldots,x_t))$, with the additional brackets dropped for aesthetic reasons. When considering \emph{unordered} $t$-tuples, the delimiter $\textrm{WL}_k(\{x_1,x_2,\ldots,x_t\})$ will always be explicitly used.
\vspace{0.5cm}

The following result relates to the conversion between the colour classes of $k$-tuples resulting from the $k$-dim WL method, and the associated colouring of $t$-tuples, for $t < k$.

\begin{theorem}
 \label{thm:macrorecursion}
 Let $S_1 = (x_1,\ldots,x_k)$ and $S_2 = (y_1,\ldots,y_k)$ be ordered $k$-tuples of $V(G)$. Then $\textrm{WL}_k(S_1) = \textrm{WL}_k(S_2)$ in $G$ only if $\textrm{WL}_k(x_{a_1},\ldots,x_{a_t}) = \textrm{WL}_k(y_{a_1},\ldots,y_{a_t})$ in $G$ for all $a_i \in [k], t < k$.
\end{theorem}

\begin{proof}
 The colouring of $t$-tuples stemming from the $k$-tuple colouring is defined to be calculated such that $t$-tuples have the same colour if and only if there are no differences between the associated $k$-tuple colour classes that could possibly distinguish them. This definition is far from explicit however, hence in what follows we will consider various \emph{possible} $t$-tuple colourings that satisfy this condition, with the aim being to settle on the simplest possible such colouring system.

 To satisfy the above condition, a colouring of $(k-1)$-tuples need only encompass the information contained in
 \begin{align*}
  \textrm{WL}_k(x_1,\ldots,x_{k-1}) = \langle \: \textrm{Sort}\{\textrm{WL}_k(x_1,\ldots,x_{k-1},i) : i \in V(G)\} \: \rangle.
 \end{align*}
 However $(k-2)$-tuples satisfy the above condition, and are hence $k$-similar if and only if more complicated sets of $k$-tuple colourings are equal, involving ordered, nested sets of $k$-tuple and $(k-1)$-tuple colourings. Simplifying the characterisation of $t$-tuple colouring will hence be potentially quite useful.

 Recall from (\ref{eq:WL_k_infty}) that
 \begin{align}
  \label{eq:WL_k-1}
  \textrm{WL}_k(x_1,\ldots,x_k) = \langle \: \textrm{Sort}\{(\textrm{WL}_k(x_1,\ldots,x_{k-1},i),\ldots,\textrm{WL}_k(i,x_2,\ldots,x_k)) : i \in V(G) \} \: \rangle,
 \end{align}
 hence the theorem holds directly for $t = k-1$.

 For $t = k-2$, a system of colouring satisfies the above condition only if $(k-2)$-tuple $\mathbf{x}$ will have colours corresponding to some ordered, nested list such as:
 \begin{align}
 \label{eq:nestedWL}
  \textrm{WL}_k(\mathbf{x}) = \langle \:\textrm{Sort}\{(&\textrm{WL}_k(\mathbf{x},i), \notag\\
  &\textrm{ Sort}\{(\textrm{WL}_k(\mathbf{x},i,j),\textrm{WL}_k(\mathbf{x},j),\ldots) : j \in V(G) \} : i \in V(G) \} \:\rangle,
 \end{align}
 where the unspecified continuation involves further $k$- and $(k-1)$ tuples involving $k-3$ elements of $(\mathbf{x})$ together with $i$ and $j$. However by (\ref{eq:WL_k-1}) $k$-tuples have equal colourings only if their corresponding ordered sets of neighbouring $(k-1)$-tuples have equal colourings, hence (\ref{eq:nestedWL}) can be simplified to:
 \begin{align}
  \textrm{WL}_k(\mathbf{x}) &= \langle \:\textrm{Sort}\{\textrm{Sort}\{\textrm{WL}_k(\mathbf{x},i,j) : j \in V(G) \} : i \in V(G) \} \:\rangle \notag\\
  &= \langle \:\textrm{Sort}\{\textrm{WL}_k(\mathbf{x},i) : i \in V(G) \} \:\rangle,
 \end{align}
 hence the theorem holds for $t = k-2$ also.

 Similarly, for general $t$ the factors of a given $t$-tuple involving $(t+i)$-tuples, where $t+i < k$ can be incorporated into the relevant $k$-tuple factors.

 Hence the $t$-tuple $\mathbf{x} = (x_1,\ldots,x_t)$ can be consistently coloured by:
 \begin{align}
  \textrm{WL}_k(\mathbf{x}) &= \langle\: \textrm{Sort} \{ \ldots \textrm{Sort} \{ \textrm{WL}_k(\mathbf{x},i_1,\ldots,i_{k-t}) : i_{k-t} \in V(G) \} \ldots : i_1 \in V(G) \} \:\rangle \notag\\
  &= \langle \:\textrm{Sort} \{ \textrm{WL}_k(\mathbf{x},i) : i \in V(G) \} \:\rangle,
 \end{align}
 without losing any relevant information present in the $k$-tuples (meaning all information that could potentially distinguish between $t$-tuples is all included).
\end{proof}

One important implication of this result is that the process of converting the $k$-tuple colourings to $t$-tuple colourings (for any $t < k$) and then \emph{back} to $k$-tuple colourings cannot partition the set of $k$-tuples further. Also note the following immediate corollary.

\begin{corollary}
\label{tuples}
 For ordered $t$-tuples $\,\mathbf{x} = (x_1,\ldots,x_t)$ and $\,\mathbf{y} = (y_1,\ldots,y_t)$ of $G$, $\;\textrm{WL}_k(\mathbf{x}) = \textrm{WL}_k(\mathbf{y})$ in $G$ only if $\textrm{WL}_k(x_{a_1},\ldots,x_{a_r}) = \textrm{WL}_k(y_{a_1},\ldots,y_{a_r})$ in $G$ for all such $r$-tuples, $r < t$, in which $a_i \in [k] \;\: \forall \; i \in [r]$.
\end{corollary}

This result also applies for $(k+i)$-tuples, where $i \ge 1$, for which a similar simplified recursive definition can be constructed.
In particular, let $S_1 = (x_1, \ldots, x_k)$ and $S_2 = (y_1, \ldots, y_k)$ be ordered $k$-tuples of $V(G)$, let $\mathbf{x}$ be an ordered $i$-tuple of $V(G)$, and let $(S_1,\mathbf{x})$, $(S_2,\mathbf{x})$ be the ordered $(k+i)$-tuples resulting from concatenating the respective $k$ and $i$ tuples. Then the following result holds.

\begin{theorem}
 \label{thm:macro_recursion_larger}
 $\textrm{WL}_k(S_1) = \textrm{WL}_k(S_2)\;$ only if $$\;\textrm{Sort} \{ \textrm{WL}_k(S_1,\mathbf{x}) : \mathbf{x} \in V(G)^{i} \} = \textrm{Sort} \{ \textrm{WL}_k(S_2,\mathbf{x}) : \mathbf{x} \in V(G)^{i} \}, \;\textrm{ for all } \; i > 0.$$
\end{theorem}

\begin{proof}
 Given some $(t+1)$-tuple $\mathbf{z} = (z_1,\ldots,z_{t+1})$, denote by $\mathbf{z}'$ the ordered set of associated $t$-tuples contained in $\mathbf{z}$, such that $\mathbf{z}' = ((z_1,\ldots,z_{t}), \ldots, (z_2,\ldots,z_{t+1}))$.
 Similarly to the proof of Theorem \ref{thm:macrorecursion} we will define the colouring of $(k+i)$-tuples such that there two $(k+i)$-tuples have the same colour if and only if no differences exist between the associated $k$-tuples that could possibly distinguish them.

 Hence a the colouring of a $(k+i)$-tuple $\mathbf{z} = (z_1,\ldots,z_{k+i})$ can be constructed recursively by:
 \begin{align}
 \label{eq:WLz}
  \textrm{WL}_k(\mathbf{z}) = \langle \textrm{WL}_k(\mathbf{z}') \rangle
 \end{align}
 Then by definition the theorem holds for $i=1$. Assume that it also holds for $i=t$, and let $\mathbf{z}$ as defined above be a $(k+t+1)$-tuple of $V(G)$. Then by (\ref{eq:WLz}) it also holds for $i=t+1$, hence by induction it holds for all $i>0$.
\end{proof}

Combining the preceding two theorems, regarding $t$-tuples where $t<k$ and $t > k$ respectively, we obtain the following corollary.

\begin{corollary}
\label{cor:extend}
 Let $\mathbf{x}$ and $\mathbf{y}$ be $t$-tuples of $V(G)$, for some $t<k$. Then $\textrm{WL}_k(\mathbf{x}) = \textrm{WL}_k(\mathbf{y})$ only if the corresponding sorted sets of $(t+i)$-tuple colours are also equal, for \emph{all} $i > 0$.
\end{corollary}

\begin{proof}
 Follows directly from Theorems \ref{thm:macrorecursion} and \ref{thm:macro_recursion_larger}.
\end{proof}

Note that the colour class of a $k$-tuple resulting from the $k$-dim WL method are linked to the paths of each length connecting elements of the $k$-tuple. In particular, Alzaga \textit{et al.} \cite{Alzaga10} show that $k$-tuples $\mathbf{x} = (x_1,\ldots,x_k)$ and $\mathbf{y} = (y_1,\ldots,y_k)$ have the same colour only if for any $i,j \in [k]$ and $m \in \mathbb{Z}$, the number of paths of length $z$ connecting $x_i$ to $x_j$, and $y_i$ to $y_j$ respectively are equal. This extends trivially to the case where the number of paths of each \emph{character} are considered, where the character of a path denotes the ordered set of colour classes associated with each element of the path in turn.

It will also be useful to establish a relationship between the colour class $c$ of a vertex $v \in V(G)$, and the properties of the set of CWS closures $\{cl(\{v,v'\}) : v' \in [c]\}$.

\begin{theorem}
 \label{thm:CWSclosures}
 Let $x_1,x_2,y_1,y_2 \in V(G)$. If $\textrm{WL}_k(x_1,x_2) = \textrm{WL}_k(y_1,y_2)$, for $k \ge 3$, then the sets $A = cl(\{x_1,x_2\})$ and $B = cl(\{y_1,y_2\})$ have the following properties.
 \begin{enumerate}
  \item $|A| = |B|$,
  \item For all $c_i \in \textrm{WL}_k(G)$, $|\{v \in A : \textrm{WL}_k(v) = c_i\}| = |\{v \in B : \textrm{WL}_k(v) = c_i\}|$,
  \item $A$ is prime if and only if $B$ is prime, and
  \item The graphs induced on $A$ and $B$ respectively are $k$-similar.
 \end{enumerate}
\end{theorem}

\begin{proof}
 Recall the notation regarding $cl(\{x_1,x_2\})$, where $S^0 = \{x_1,x_2\}$, $R(x,y) = (d(x) \cap e(y)) \cup (e(x) \cap d(y))$ and $S^{i+1} = (\bigcup_{x,y}R(x,y)) \cup S^i$.

 For $u,v \in V(G)$, let $u \in S^1$, and $v \notin S^1$. Then $\textrm{iso}(x_1,x_2,u) \neq \textrm{iso}(x_1,x_2,v)$ and hence $\textrm{WL}_k(x_1,x_2,u) \neq \textrm{WL}_k(x_1,x_2,v)$.

 Further, let $u \in S^{t+1}$ such that $u \notin S^t$, and let $v \notin S^t$. Then
 \begin{align}
  &\exists \; i,j \in S^t \textrm{ such that } u \in R(i,j), \textrm{ and} \notag\\
  &\nexists \; i,j \in S^t \textrm{ such that } v \in R(i,j).
 \end{align}
 Assume that for all $l \in S^t, m \notin S^t$, we have $\textrm{WL}_k(x_1,x_2,l) \neq \textrm{WL}_k(x_1,x_2,m)$. Note that if $k \ge 3$ we then have $\textrm{WL}_k(x_1,x_2,u) = \textrm{WL}_k(x_1,x_2,v)$ only if there exist some $l' \in S^t$ and $m' \notin S^t$ such that
 \begin{align}
  &\textrm{WL}_k(l,u) = \textrm{WL}_k(l',v) \textrm{, and} \notag\\
  &\textrm{WL}_k(m,u) = \textrm{WL}_k(m',v).
 \end{align}
 Hence $u \in S^{t+1}$ such that $u \notin S^t$, and $v \notin S^t$ implies that $\textrm{WL}_k(x_1,x_2,u) \neq \textrm{WL}_k(x_1,x_2,v)$.

 Since $u \in S^1$, and $v \notin S^1$ implies that $\textrm{WL}_k(x_1,x_2,u) \neq \textrm{WL}_k(x_1,x_2,v)$, then by induction this holds for all $t > 1$ also.

 Furthermore, by corollary \ref{tuples}, for all $u,v \in S^t$, $\textrm{WL}_k(x_1,x_2,u) = \textrm{WL}_k(x_1,x_2,v)$ only if $\textrm{WL}_k(u) = \textrm{WL}_k(v)$, hence (2) holds (and as a corollary, (1) holds).

 Let $i,j \in A$, $\textrm{WL}_k(i) = \textrm{WL}_k(j)$, such that $cl(\{i,j\}) \neq A$ (i.e. $A$ is not prime). Then there exists a vertex $u \in A$, $u \notin cl(\{i,j\})$. Assume further that $B$ is prime, and so $\nexists \; l,m \in B$, $\textrm{WL}_k(l) = \textrm{WL}_k(m)$ such that $v \notin cl(\{l,m\})$ for some $v \in B$. Hence $\nexists \; l,m \in B$ such that $\textrm{WL}_k(l,m) = \textrm{WL}_k(i,j)$, from which it follows (from corollary \ref{tuples} and the proof of (2)) that $\textrm{WL}_k(x_1,x_2) \neq \textrm{WL}_k(y_1,y_2)$. Hence (3) holds.

 Similarly, let $\mathbf{z}$ be a $t$-tuple of $V(G)$. By Corollary \ref{cor:extend}, $(x_1,x_2)$ and $(y_1,y_2)$ belong to the same colour class only if 
 $$\textrm{Sort} \{ \textrm{WL}_k(x_1,x_2,\mathbf{z}) : \mathbf{z} \in V(G)^{t} \} = \textrm{Sort} \{ \textrm{WL}_k(y_1,y_2,\mathbf{z}) : \mathbf{z} \in V(G)^{t} \},$$
 for all $t > 0$. In particular, note that for any $\mathbf{z_1} \in A$ and $\mathbf{z_2} \notin A$, we have
 $$\textrm{WL}_k(x_1,x_2,\mathbf{z_1}) \neq \textrm{WL}_k(x_1,x_2,\mathbf{z_2}).$$
 Hence
 $$\textrm{Sort} \{ \textrm{WL}_k(x_1,x_2,\mathbf{z}) : \mathbf{z} \in A \} = \textrm{Sort} \{ \textrm{WL}_k(y_1,y_2,\mathbf{z}) : \mathbf{z} \in B \},$$
 and in particular,
 $$\textrm{WL}_k(A) = \textrm{WL}_k(B).$$
 Finally, note that if two CWS subgraphs of $G$ are $k$-similar \emph{within} $G$, then the respective induced graphs are also $k$-similar, from which (4) follows.
\end{proof}

\begin{definition}
 Let $v \in [c]$, such that $c = \textrm{WL}_k(v) \mid_G$. Denote the \emph{CWS spectrum} of a vertex $v$ to be the set of pairwise CWS closures $C_v = \{cl(\{v,v'\}) : v' \in [c]\}$.
\end{definition}

Then the following corollary of Theorem \ref{thm:CWSclosures} holds.

\begin{corollary}
\label{spectrum}
 Let $u,v \in V(G)$. Then $\textrm{WL}_k(u) = \textrm{WL}_k(v)$ only if there is a matching between elements of $C_v$ and $C_u$ in the sense of Theorem \ref{thm:CWSclosures}.
\end{corollary}

\begin{example}
 Consider the case where for $u,u' \in [c]$ there is a unique prime closure $A = cl(\{u,u'\})$ such that if $B = cl(\{u,v\})$ is prime for some $v \in [c]$, then $A = B$. Then if $C= cl(\{x,x'\})$ is prime for $x,x' \in [c]$ then by Theorem \ref{thm:CWSclosures} $C$ must also unique in this sense.
\end{example}

\setcounter{equation}{0}
\section{Extended modular decomposition method}
\label{sec:decomposition}

At this point we note that the only currently known graphs for which the recursive $k$-dim WL method fails (in that it fails to recursively partition the relevant vertex set to its orbits) belong to $M_k$, containing non-isomorphic $k$-equivalent subgraphs. In particular, to the knowledge of the author, $k$-equivalent graphs of the general type in \cite{CFI} and \cite{Evdokimov99} are the only graphs for which the recursive 3-dim WL method is known to fail, in that the only known pairs of non-isomorphic 3-isoregular graphs with the same parameters are vertex-transitive.

With this in mind, in proposing a method of dealing with these specific counterexample graphs, we will assume for the purposes of this paper (and in particular, the following proposed decomposition method) that \emph{prime} $k$-equivalent graphs are characterised by the recursive $(k+1)$-dim WL method (where $k \ge 3$). One reason why this assumption might not be considered particularly onerous for the purposes of this paper is that, as shown in Chapter \ref{sec:ext}, if the original expander graph used to construct the counterexample pairs in \cite{CFI} and \cite{Evdokimov99} can be recursively partitioned down to its orbits by the $k$-dim WL method, then the $(k+1)$-dim WL method achieves also this for each of the counterexample pairs themselves.

\subsection{Preliminary definitions}

\begin{definition}
 The \emph{extended modular decomposition method} is defined as a process of isolating relevant CWS subgraphs of a graph.
\end{definition}

In particular, the aim is to isolate then characterise the mutually CWS, $k$-equivalent subgraphs, these being the components that provably cannot be partitioned to their orbits by the $k$-dim WL method. The title of this section stems from the analogous definition of a modular decomposition of a graph. The modules of a graph are subgraphs within which each element has the same set of neighbours among elements outside the module. Modules can be proper subsets of other modules, hence the term leads to a recursive decomposition of a graph, with the set of modules of a graph forming a lattice under inclusion. 

Similarly the set of CWS subsets of a graph also forms a lattice with respect to inclusion, as shown below, and can be thought of as a generalisation of the idea of modules, in this case relative to the colour classes assigned by the $k$-dim WL method. In this generalisation, only elements of the CWS subgraph of the same colour class are required to have identical neighbour sets outside the subgraph.

Relative to the standard definition of modules, the modular closure of set of vertices $S \in V(G)$ is defined to be the smallest module $R \in V(G)$ that contains $S$. Here, the term modular closure will instead be defined relative to this generalisation of modules, according to the following definition.

\begin{definition}
 The \emph{modular closure}, or simply \emph{closure}, of a set of vertices $S \in V(G)$, denoted $cl(S)$, is the smallest CWS subset of $G$ containing $S$ (i.e. the supremum of $S$, relative to CWS modules).
\end{definition}

This modular closure is defined relative to the colour classes arising from the $k$-dim WL method. There is a simple procedure to calculate the modular closure of a set $S$, introduced in Chapter \ref{sec:WLprops}. Consider any two elements $u,v \in S$ in the same colour class of $\textrm{WL}_k(G)$. Then the elements $(d(u) \cap e(v)) \cup (e(u) \cap d(v))$ must also be in $cl(S)$. Recursively performing this process until membership in $cl(S)$ is stabilised yields a unique $cl(S)$. Hence $cl(S)$ is well defined.

\begin{definition}
 A \emph{non-trivial} CWS subset is defined as one containing at least two elements of the same colour class.
\end{definition}

\begin{observation}
 The CWS subsets of a graph form a lattice, under inclusion.
\end{observation}

\begin{proof}
 Consider a graph $G$, containing two CWS subsets $A$ and $B$. The modular closure $cl(A \cup B)$ defines a unique supremum (in terms of CWS subsets). Consider the intersection of the two subsets, $C = (A \cap B)$. Then $cl(C) \subseteq A$, $cl(C) \subseteq B$, and $cl(C) = sup(C)$, hence the modular closure $cl(A \cap B)$ defines a unique infimum.
\end{proof}

Recall the primality definition given in Section \ref{sec:orbit_counter}.

\begin{repdefinition}{def:prime}
 A subgraph $S \subseteq G$ will be termed \emph{prime} if it has no proper, non-trivial CWS subgraphs, and is itself a non-trivial CWS subgraphs of $G$. This is defined implicitly with respect to the $k$-dim WL method.
\end{repdefinition}

\begin{definition}
 Given a $k$-equivalent graph $G$ with colour classes corresponding to $\textrm{WL}_k(G)$, a set of colour classes $C = \{c_1, c_2,\ldots \}$ of $G$ will be termed to be \emph{trivial} in $G$ if the graph induced on $cl([C])$ is not $k$-equivalent.
\end{definition}

\begin{definition}
 Denote a CWS subset $S$ of $G$ to be non-trivial \emph{relative} to the colour class $c$ if $S$ contains more than one vertex belonging to $[c]$.
\end{definition}

\subsection{Graphs under consideration}

Recall that the set of graphs $M_k$ is defined in terms of $k$-equivalent subgraphs $S_i \subset G$, that are CWS connected with respect to the colour classes of $\textrm{WL}_{k}(S_i) \mid_{S_i}$, rather than the colour classes of $\textrm{WL}_{k}(S_i) \mid_G$. Hence it is possible that distinct colour classes of such an induced graph $S_i$ will be \emph{merged} in $G$, in that $\exists \, u,v \in V(S_i)$ such that
\begin{align*}
 &\textrm{WL}_k(u) \mid_{S_i} \neq \textrm{WL}_k(v) \mid_{S_i}, \textrm{ but} \\
 &\textrm{WL}_k(u) \mid_{G} = \textrm{WL}_k(v) \mid_{G}.
\end{align*}
In order to simplify the analysis of the decomposition method that follows, it will be defined to act on a subset of $M_k$ in which the properties of $\textrm{WL}_{k}(G)$ are constrained relative to the colour classes associated with the induced graphs $S_i$.

In particular, it will be defined to act on the set of graphs $M^{'}_k \subset M_k$ defined as follows.

\begin{definition}
 \label{assume:M}
 $M^{'}_k$ consists of the graphs $G \in M_k$ for which the following further properties hold.
 \begin{enumerate}
  \item Consider a set of vertices $\{v_1,\ldots,v_r\} \subset V(G)$ with the properties
  \begin{align}
   &\textrm{WL}_k(v_i) = c, \forall i \in [r]\textrm{, and} \notag\\
   &d(v_i) \backslash \{v_j\} = d(v_j) \backslash \{v_i\} \forall i, j \in [r].
  \end{align}
  No such set of `CWS cliques' exist in $G$.
  \item Prime CWS subgraphs of $G$ are unique, in the sense that for any $x,y,z$ belonging to the same colour class of $G$, $cl(\{x,y\})$ and $cl(\{x,z\})$ are both prime only if $cl(\{x,y\}) = cl(\{x,z\})$.
 \end{enumerate}
\end{definition}

The first condition removes the possibility of $G$ containing modules in which all elements belong to the same colour class, and the graph induced on the module is either complete or empty. Every subset of such a module is also a CWS subset of $G$. The first and second conditions are included to simplify the analysis of the decomposition algorithm presented later in this section. While they are listed as assumptions, we will see that both can be enforced without loss of generality, by canonically altering a given graph in $M_k$.

Before demonstrating that the first and second assumptions can be assumed to hold without loss of generality, further definitions will be required. Consider the following process of \emph{canonically contracting} a CWS subgraph $S$ of a graph $G$, replacing $S$ by a single vertex, with the resulting graph labelled by $G'$.

\begin{definition}[\textbf{Canonical Contraction}]
\label{def:contract}
 For each colour class $c_i \in \textrm{WL}_k(G)$ with members in $S$, denote
 $$[c_i]_S = \{ v \in [c_i] : v \in S \}.$$
 Since $S$ is a CWS subgraph of $G$, a vertex $w \in V(G \backslash S)$ can be said to be connected to a colour class of $S$, in that if $\{v,w\} \in E(G)$ for some $v \in [c_i]_S$, then $\{v',w\} \in E(G)$ for all $v' \in [c_i]_S$. 

 \noindent For each $w \in V(G \backslash S)$, denote by $w_S = \{(c_r,\gamma_1), \ldots, (c_t, \gamma_t)\}$ the set of colour classes $c_i$ such that $w$ is connected to $[c_i]_S$ by edges of colour $\gamma_i$.
 The \emph{canonical contraction} of $G$ relative to a subgraph $S \subset G$, resulting in a graph $G^{'}_S$, proceeds as follows.
 \begin{itemize}
  \item[(i)] Replace $S$ by a single vertex $x$ coloured by the isomorphism class of $S$.
  \item[(ii)] Replace the edges connecting $[c_i]_S$ to $w \in V(G \backslash S)$ by a single edge $\{x,w\}$ coloured by $c_i$, for all such $w$.
  \item[(iii)] Where $w \in V(G \backslash S)$ is connected to multiple colour classes in $S$, replace the resulting multiple edges by a single edge coloured by each such colour class, such that:
  \item[(iv)] Edges $\{w,x\}$ and $\{w^{'},x\}$ have the same colour if and only if $\textrm{Sort}(w^{}_S) = \textrm{Sort}(w'_S)$.
 \end{itemize}
\end{definition}

Hence this contraction replaces a CWS subgraph $S$ by a single vertex, while preserving all information regarding the isomorphism class of $S$ and its connections to $G \backslash S$, such that given graphs $G$ and $H$ with subgraphs $S_1$ and $S_2$ canonically contracted respectively, $G' \cong H'$ if and only if
$$S_1 \cong S_2 \;\textrm{ and }\; G \cong H.$$
The contraction is then unbiased (or canonical) in an analogous sense to the \emph{unbiased extension} of Chapter \ref{sec:ext}.

Given a graph $G \in M_k$, 
consider the following two contractions of $G$. 

\begin{construction}
\label{constr:1}
Firstly, for each colour class $c_i \in \textrm{WL}_k(G)$, let $\mathcal{S} = \{S_1,\ldots,S_t\}$ be the set of all \emph{maximum CWS cliques} in $G$ for which all subsets of each $S_i$ are also CWS subgraphs of $G$ (in other words, precisely the CWS subgraphs described in (1) of Definition \ref{assume:M}), where each $S_i$ is maximal in the sense that no $S^{'}_i \supset S_i$ exists with the same property . Note that the set of such cliques can be efficiently found for any graph $G$. Replace all such $S_i$ by a single vertex as in Definition \ref{def:contract}, repeat this process recursively until the result is stabilised, and label the resulting graph by $G_1$. 
\end{construction}

\begin{definition}
 \label{def:contract1}
 Given a graph $G$, $G_1$ is the resulting graph in which all \emph{maximum CWS cliques} are recursively contracted to single vertices as in Construction \ref{constr:1} above, such that $G_1$ contains no such subgraphs.
\end{definition}

\begin{construction}
\label{constr:2}
Secondly, given a subgraph $S \subset G_1$, let $[c_i]_S = \{ v \in [c_i] : v \in S \}$ denote the set of vertices of $S \subset G_1$ belonging to colour class $c_i$. Consider a CWS subgraph $R \in G_1$ with the following properties. For all $i,j$, where $v \in [c_i]_R$ and $c_j \in \textrm{WL}_k(G_1)$,
\begin{align}
 &d(v) \cap [c_i]_R = [c_i]_R \textrm{ or } \emptyset, \notag\\
 &|d(v) \cap [c_j]_R| = 0, \;\; 1, \;\; |[c_j]_R|-1 \textrm{ or }\; |[c_j]_R|.
\end{align}
Such a CWS subgraph $R$ has the property that any prime CWS subgraphs of $R$ have exactly two elements from each colour class in $R$. Any two prime CWS subgraphs of $R$ are either vertex disjoint, equal, or have an intersection comprising exactly one element from each colour class in $R$. Further, given non-equal prime CWS subgraphs $cl(\{x,y\})$ and $cl(\{x,z\})$ in $R$, where $x,y,z \in [c_i]_R$ for some $i$, the intersection $R_x = cl(\{x,y\}) \cap cl(\{x,z\})$ is independent of the particular $y,z$ chosen. Finally, the set $\{R_x : x \in [c_i]_R\}$ partitions $R$. Note that the set of subgraphs $R$ satisfying the above properties can be efficiently found for any graph $G$, and this set forms a lattice in $G$ with respect to inclusion.

Replace each such vertex-disjoint $R_x$, belonging to such a subgraph $R \subset G_1$, with a single vertex as in Definition \ref{def:contract}, repeat this process recursively until the result is stabilised, and label the resulting graph by $G_2$.
\end{construction}

\begin{definition}
 \label{def:contract2}
 Given a graph $G_1$, $G_2$ is the resulting graph in which all subgraphs of the type described above (and those associated with $G_1$) are recursively contracted as in Construction \ref{constr:2}, such that $G_2$ contains no such subgraphs.
\end{definition}

\begin{theorem}
Let $G$ be a graph in $M_k$, and let $G_1$ and $G_2$ be the graphs resulting from $G$ by the contractions described above. Then the following statements hold. 
 \begin{itemize}
  \item[(i)] $G_1$ satisfies condition (1) of Definition \ref{assume:M}.
  \item[(ii)] $G_2$ satisfies conditions (1) and (2) of Definition \ref{assume:M}.
 \end{itemize}
\end{theorem}

\begin{proof}
 By definition, if $G_1$ contains such CWS cliques, then $G_1$ is not stable, in that the contraction described can be applied to $G_1$ resulting in some graph $(G_1)_1 \neq G_1$. Hence (i) holds.

 Let $A = cl(\{x,y\})$, $B = cl(\{x,z\})$ be prime CWS subgraphs of $G_2$, where $x,y,z \in [c_i]$ for some colour class $c_i \in \textrm{WL}_k(G_2)$. 

 Let $\{x,y\} \notin cl(A \backslash \{x,y\})$. Then $A \backslash \{x,y\}$ contains at most one vertex of each colour class, else $A$ is not prime. Similarly, either $A = \{x,y\}$ or there exists a $v \in A, v \notin \{x,y\}$ such that $v \in d(x), v \in e(y)$ or $v \in e(x), v \in d(y)$. Since $A$ is prime, $d(x) \cap \overline{A} = d(y) \cap \overline{A}$, however this contradicts the assumption that $x,y \in [c_i]$, since $x$ and $y$ must have a different number of neighbours in the colour class of $v$. Hence no such $v$ exists, and $A = \{x,y\}$. As no such CWS cliques exist in $G_1$, this is a contradiction, and so we must have $\{x,y\} \in cl(A \backslash \{x,y\})$.

 Now $A \cap B$ contains at most one vertex from each colour class in $A$ and $B$ (else $A,B$ are not prime). Since $x,y,z \in [c_i]$, for any colour class $c_j$, $|d(x) \cap [c_j]| = |d(y) \cap [c_j]| = |d(z) \cap [c_j]|$. In particular, $x$ and $y$ have the same connections outside $A$, $x$ and $z$ have the same connections outside $B$, and $y$ and $z$ have the same connections outside $(A \cup B) \backslash (A \cap B)$. Denote $[c_j]_A = [c_j] \cap A$. Then $d(x) \cap [c_j]_A = d(y) \cap [c_j]_A$ for each colour class $c_j$, and similarly for $x,z$ in $B$ and $y,z$ in $(A \cup B) \backslash (A \cap B)$.

 Consider a colour class $c_j$ with elements in $A \cup B$, where $J = [c_j]_{A \cup B}$. 
 Let $J' = J \cap (A \cap B)$. If $J' \neq \emptyset$, then $J' = \{v\}$ for some $v \in (A \cap B)$. Then either $J \backslash \{v\} \subset d(v)$, or $J \backslash \{v\} \subset e(v)$, and hence the graph induced on $J$ is either complete or empty. Trivially, $|J \cap A| = |J \cap B|$ holds.

 Consider a second colour class $c_l$, where $L = [c_l]_{A \cup B}$, such that $c_l$ also has non-empty intersection, $L \cap (A \cap B) = v'$. As in the connections \emph{within} $[c_j]_{A \cup B}$, if $v' \in d(u)$ for some $u \in [c_j]_A$, then $L \cap B \subset d(u)$, which in turn implies $(L \cap B) \backslash \{v'\} \subset d(u')$ for all $u' \in [c_j]_A$, and $v' \in d(u')$ for all $u' \in ([c_j]_A \backslash \{v\}$. Hence $|d(v') \cap J| \ge |J|-1$, implying that $|d(w) \cap J| \ge |J|-1$ for all $w \in K$.

 So far we have established a basic structure for the connections between and within colour classes of $A$ and $B$. Namely, if $A \cap B$ contains a vertex $v \in [c_j]$, then the graph induced on $[c_j]_{A\cup B}$ is either complete or empty. Further, if $v' \in [c_l]$ such that $v' \in (A \cap B)$, then the connections \emph{between} $[c_j]$ and $[c_l]$ within $A \cup B$ are either uniform or `almost uniform', in that:
 \begin{align}
  &\forall u \in J, |d(u) \cap L| = |L| \textrm{ or } |L-1|, \textrm{ and hence} \notag\\
  &\forall u \in L, |d(u) \cap J| = |J| \textrm{ or } |J-1|.
 \end{align}
 Given these constraints on $A\cup B$, consider the iterative process of constructing $A$ and $B$ from $\{x,y\}$ and $\{x,z\}$ respectively, described in Chapter \ref{sec:WLprops}. At the first step, consider the set $D = \{v_1,\ldots,v_t\}$, where for all $v_i \in D$,
 \begin{align}
  v_i \in (d(x) \cap e(y)) \textrm{ or } v_i \in (d(y) \cap e(x)).
 \end{align}
 Let $v_i \in [c_j]$ for some $v_i \in D$. As the connections between $\{x,y\}$ and $[c_j]_A$ are uniform or almost uniform in the above sense, either $d(x) \cap [c_j] = d(y) \cap [c_j]$ or $|D \cap [c_j]| = 2$, one vertex of which necessarily belongs to $A \cap B$. Iterating this process, if $|[c_m] \cap (A \cup B)| > 0$ for any colour class $c_m$, then
 \begin{align}
  &|[c_m] \cap (A \cup B)| = 3, \notag\\
  &|[c_m] \cap (A \cap B)| = 1, \textrm{ and}\\
  &|[c_m] \cap A| = |[c_m] \cap B| = 1. \notag
 \end{align}
 Hence $A \cap B$ is precisely the (trivial) CWS subgraph type that is contracted in the process described in \ref{def:contract2} above. As no such subgraphs exist in $G_2$, no such intersecting prime CWS subgraphs $cl(\{x,y\})$ and $cl(\{x,z\})$ exist, and hence prime CWS subgraphs of $G_2$ are in this sense unique.
\end{proof}

\begin{corollary}
 Any $G \in M_k$ can be assumed without loss of generality to satisfy conditions (1) and (2) of Definition \ref{assume:M}, by instead considering the graph $G_2$ obtained by recursively applying the two contraction process to $G$ described above.
\end{corollary}

Note that the graphs of interest within $M^{'}_k$ contain those known to not be partitioned down to their orbits by the $k$-dim WL method. For such a graph $G$, there will exist at least one colour class $c \in \textrm{WL}_k(G)$ consisting of two or more orbits of Aut($G$). 

\begin{lemma}
\label{thm:notprime}
 Let $G \in M^{'}_k$ contain two vertex-disjoint, non-trivial CWS, mutually $k$-similar subgraphs $S_1, S_2$, such that $u \sim_k v$ within $G$, for some $u \in V(S_1)$, $v \in V(S_2)$. Then $cl(\{u,v\})$ is not prime.
\end{lemma}

\begin{proof}
 Assume that $cl(\{u,v\})$ is prime, and let $A = cl(\{u,v\}) \cap S_1$ and $B = cl(\{u,v\}) \cap S_2$. $A$ and $B$ are each CWS subgraphs of $G$, so either $cl(\{u,v\})$ is not prime, or $A$ and $B$ each contain at most one vertex from any given colour class.

 Let $u,v \in [c]$ for some colour class $c \in \textrm{WL}_k(G)$. Then for all $x \in A, x \in [c_i]$ there exists a $y \in B$ such that $y \in [c_i]$ and vice versa. Hence $A$ and $B$ are mutually $k$-equivalent. Let $R_1(i)$ and $R_2(i)$ represent the sets of vertices of colour class $c_i$ in $S_1$ and $S_2$ respectively. Since $S_1$ and $S_2$ are each CWS subsets of $G$,
 \begin{align}
  d(u) \cap R_2(i) = R_2(i) \textrm{ or } \emptyset,
 \end{align}
 and similarly for $d(v)$. As $|cl(\{u,v\}) \cap R_1(i)| = 1 \textrm{ or } 0$, then
 \begin{align}
  d(u) \cap R_1(1) &= R_1(1) \backslash u \textrm{ or } \emptyset \textrm{, and} \notag\\
  |d(u) \cap R_1(i)| &= |R_1(i)|, 0, 1 \textrm{ or } |R_1(i)|-1,
 \end{align}
 and similarly for $d(v)$ in $S_2$. Note that since $S_1$ and $S_2$ are mutually $k$-equivalent,
 \begin{align}
  |d(u) \cap R_1(i)| = |d(v) \cap R_2(i)| \;\forall i.
 \end{align}
 Hence if $S_1$ contains another vertex $w \neq x$ of colour class $[c]$, $cl(\{x,w\})$ is also prime, and the sets $A$ and $B$ are precisely those contracted to a single vertex by the contraction process of Definition \ref{def:contract2}. Hence no such prime closure $cl(\{x,y\})$ exists.
\end{proof}

\subsection{Decomposition method}

Following the above introductory definitions and properties, we can now define the following method of decomposing $G \in M^{'}_k$ into the mutually CWS, mutually $k$-equivalent $S_i$ subsets that are prime.

\begin{algorithm}[\textrm{Decomposition}]
\label{alg:decomposition}
  Given a graph $G \in M^{'}_k$, we define the following extended modular decomposition method:
 \begin{enumerate}
  \item Act on $G$ with the $k$-dim WL method, determining the colour classes $\textrm{WL}_k(G)$ (relative to which the CWS subsets are defined).
  \item Choose a vertex $u$ (without loss of generality) from the lexicographically smallest colour class, $c$.
  \item For each vertex $v \in [c], v \neq u$, calculate $cl(\{u,v\})$, and determine if this closure is prime.
  \item If no such prime closure exists, remove this colour class from consideration, returning to step (3). Else record the unique prime closure $cl(\{u,v\})$ (for some appropriate $v$).
  \item Repeat steps (2)-(4) recursively on the remaining elements of $\left[ c \right]$ (not currently contained in a prime closure) until all elements of $[c]$ are associated with a prime CWS subgraph.
  \item Repeat steps (2)-(5) recursively for the remaining colour classes not incorporated in the previous prime closures. 
 \end{enumerate}
\end{algorithm}

\begin{lemma}
\label{thm:assorted}
 Let $G \in M^{'}_k$ be a graph as in definition \ref{alg:decomposition} above. For each colour class $c \in \textrm{WL}_k(G)$, the following hold:
 \begin{itemize}
  \item[(i)] There exists a partitioning $\{V_1,V_2,\ldots\}$ of $[c]$ such that all $cl(V_i)$ are prime, and all $cl(V_i)$ subgraphs are mutually $k$-similar.
  \item[(ii)] If a $cl(V_i)$ as above contains vertices of colour class $c'$, then $\{cl(V_1),cl(V_2),\ldots \}$ partitions $[c']$ in the same manner.
 \end{itemize}
\end{lemma}

\begin{proof}
Let $u,v \in [c]$ such that $cl(\{u,v\})$ is prime. Then for all $w \in [c]$, $cl(\{u,w\})$ is prime if and only if $w \in cl(\{u,v\})$. Similarly, by Theorem \ref{thm:CWSclosures}, each vertex $u \in [c]$ has sets of pairwise CWS closures sharing several properties. In particular, for all $u \in [c]$ there is a unique prime closure containing $u$, and furthermore for any $u,v \in [c]$, the prime closures containing $u$ and $v$ respectively are either $m$-equivalent (for some $m \ge k$) or isomorphic. Hence (i) holds, and as a corollary, (ii) holds.
\end{proof}

\begin{corollary}
 Any two prime CWS subgraphs resulting from Algorithm \ref{alg:decomposition} which have overlapping colour classes are $k$-similar. 
\end{corollary}

The set of colour class partitions outputted by the above decomposition method are used to characterise the graph $G \in M^{'}_k$, according to the following `wrapper' algorithm.

\begin{algorithm}[\textbf{Reduction}]
 \label{alg:wrapper}
 Denote the set of prime, vertex-disjoint, CWS subgraphs obtained from the method of Algorithm \ref{alg:decomposition} by $\mathcal{T}_G$.
 \begin{enumerate}
  \item Apply the recursive $k$-dim WL method to each $T \in \mathcal{T}_G$, obtaining a graph certificate characterising each\footnote{By assumption, prime $k$-equivalent graphs are characterised by the recursive $k$-dim WL method. Further, we assume that non-$k$-equivalent graphs are also characterised by the recursive $k$-dim WL method.}.
  \item Canonically contract each $T \subset G$ as described in Definition \ref{def:contract}, replacing it with a single vertex  $v_T$ coloured by the isomorphism class of $T$, and similarly contracting the edges incident with $T$ as in Definition \ref{def:contract}.
  \item Label the resulting graph $G^{(1)}$ (where $G := G^{(0)}$).
  \item Recursively repeat steps (1)-(3), obtaining the graph $G^{(i)}$ after the $i^{\textrm{th}}$ repetition, until the resulting graph is stabilised, such that $G^{(t)} = G^{(t+1)}$.
 \end{enumerate}
\end{algorithm}

\begin{theorem}
 Applying the process of Algorithms \ref{alg:decomposition} and \ref{alg:wrapper} to graphs $G$ and $H$ in $M^{'}_k$,
 $$G^{(i)} \cong H^{(i)} \textrm{ if and only if } G \cong H.$$
\end{theorem}

\begin{proof}
 As the prime CWS subgraphs of a graph $G$ are unique (in the sense of Definition \ref{assume:M}), $G \cong H$ only if $\mathcal{T}_G = \mathcal{T}_H$. Hence, as the contraction process of Definition \ref{def:contract} preserves isomorphism, 
 $G^{(1)} \cong H^{(1)}$ if and only if $G \cong H$. Similarly $G^{(i)} \cong H^{(i)}$ if and only if $G^{(i-1)} \cong H^{(i-1)}$, and the result follows.
\end{proof}

\begin{theorem}
 If $G^{(t)} = G^{(t+1)}$ then $G^{(t)}$ is characterised by the recursive $k$-dim WL method.
\end{theorem}

\begin{proof}
 Let $G^{(t)}$ contain some CWS $k$-equivalent subgraph $A$. If $A$ is not prime, then there exists some prime CWS subgraph $S \subset A$ such that $S \in \mathcal{T}_{G^{(t)}}$. However this implies that  $G^{(t+1)} \neq G^{(t)}$, hence no such $S$ exists, and $A$ is prime.

 Hence if any subgraph $A \subseteq G^{(t)}$ is $m$-equivalent for some $m \ge k$, then $A$ is prime. By assumption, prime $k$-equivalent graphs are characterised by the recursive $k$-dim WL method, and so the result follows.
\end{proof}

\begin{definition}
 The smallest $t$ such that $G^{(t)} = G^{(t+1)}$ is termed the \emph{recursion depth} of $G$.
\end{definition}

\begin{lemma}
 For any $G \in M^{'}_k$, the following hold:
 \begin{itemize}
  \item[$(i)$] $G$ has recursion depth bounded above by $O(\textrm{log } |V(G)|)$. 
  \item[$(ii)$] $G^{(i)}$ can be calculated from $G^{(i-1)}$ in time $O(\textrm{poly}(|V(G^{(i)})|))$.
 \end{itemize}
\end{lemma}

\begin{proof}
 For (i), note that the elements of $\mathcal{T}_G$ partition the colour classes of interest. Moreover, elements of $\mathcal{T}_{G^{(1)}}$ must contain vertices corresponding to the elements of $\mathcal{T}_G$, and two such vertices $v_1,v_2$ are in the same colour class if and only if the corresponding $T_1,T_2 \in \mathcal{T}_G$ are isomorphic and mutually CWS. Now each $T \in \mathcal{T}_G$ contains at least 3 vertices, so similarly each $T \in \mathcal{T}_{G^{(1)}}$ contains at least two vertices corresponding to mutually CWS $T \in \mathcal{T}_G$, and hence corresponds to at least 7 vertices of $G$. Hence each $T \in \mathcal{T}_{G^{(i)}}$ corresponds to at least $2^{i+2} - 1$ vertices of $G$, and (i) follows.

 To show that (ii) holds, note briefly that each step in Algorithm \ref{alg:decomposition} can be done in time $O(\textrm{poly}(|V(G^{(i)})|))$. Namely, Algorithm \ref{alg:decomposition} consists of first applying the $k$-dim WL method, requiring time $O(n^{k+1})$ for a graph on $n$ vertices, followed by calculating the modular closure of at most $O(n^2)$ pairs of vertices (and checking each closure for primality), which can also be accomplished in polynomial time. The implicit extra step of contracting requisite subgraphs of $G \in M_k$ to form a graph in $M'_k$, as detailed in Constructions \ref{constr:1} and \ref{constr:2} is also trivially accomplished in time $O(\textrm{poly}(n))$, for a graph of size $n$, hence (ii) follows.


\end{proof}

\begin{corollary}
 Hence the combination of Algorithms \ref{alg:decomposition} and \ref{alg:wrapper} characterises the graphs of $M_k$ that satisfy the assumptions recalled in \ref{sec:assumptions} in polynomial time.
\end{corollary}

The significance of these results, and of the assumptions made, will be discussed in the following sections.

\subsection{Discussion of Assumptions}
\label{sec:assumptions}

The decomposition method comprising Algorithms \ref{alg:decomposition} and \ref{alg:wrapper} relies on assumptions regarding the properties of the input graphs, relative to the $k$-dim WL method. Specifically, any graph $G \in M^{'}_k$ will be characterised by Algorithm \ref{alg:wrapper} providing the following assumptions hold:

\begin{enumerate}
\label{list:assumptions} 
 \item Non-$k$-equivalent graphs are characterised by the recursive $k$-dim WL method (where $k \ge 3$).
 \item Prime $k$-equivalent graphs can be characterised by the recursive $(k+1)$-dim WL method.
\end{enumerate}

Non-$k$-equivalent graphs are precisely those which are assigned unique certificates (and hence characterised) by the $k$-dim WL method. This does not necessarily imply that the recursive $k$-dim WL method will characterise such graphs, in that they may not be refined to their orbits at each step. In particular, $3$-isoregular graphs that aren't vertex-transitive, or $4$-isoregular graphs that aren't $2$-transitive will not be refined to their orbits by the $3$-dim and $4$-dim WL method respectively. Whether this assumption holds in general is unknown, however note that as shown in Chapter \ref{sec:ext}, the known $k$-equivalent graphs are characterised by the recursive $(k+1)$-dim WL method, and hence no counterexamples to assumption (1) are yet known.

Similarly the only known prime $k$-equivalent graphs are those described in \cite{CFI} and \cite{Evdokimov99}, and so likewise no counterexamples to assumption (2) are known. One possible method of constructing prime $k$-equivalent graphs which do not satisfy this assumption might be extending a $k$-equivalent graph as described in Chapter \ref{sec:ext}, using unbiased gadgets with vertices corresponding to $3$-tuples of the original graph. In which case the $(k+2)$-dim WL method suffices to refine the resulting graph to its orbits, although it may not be required.

\setcounter{equation}{0}
\section{Conclusions and Future Work}
\label{sec:conclusions}

The goal of this work is to explore the question,
\vspace{-0.3cm}
\begin{center}
 \textit{``Has the $k$-dim WL method been proven to not solve GI?''}
\end{center}
\vspace{-0.3cm}
\noindent This has indeed been proven regarding a direct implementation of the $k$-dim WL method (in \cite{CFI, Evdokimov99}). However one of the essential characteristics of these proofs are that the counterexample graphs found (termed $k$-equivalent graphs) possess very specific properties. As a result, this work explores the possibility of exploiting these restrictive properties to design an extension to the WL method that can characterise these graphs.

We show that the $k$-equivalent graphs constructed in \cite{CFI, Evdokimov99} are individually characterised by the recursive $(k+1)$-dim WL method. These results are expanded on, constructing a family of $k$-equivalent graphs not refined to their orbits by the recursive $k$-dim WL method. Given this family of graphs, we establish various related properties of the $k$-dim WL method, and construct and algorithm that canonically decomposes such $k$-equivalent graphs into graphs for which the recursive $k$-dim WL method succeeds, in the process characterising the original graph.

In the process an extension to the $k$-dim WL method is constructed which efficiently characterises the known $k$-equivalent graphs, and hence represents a potential candidate for solving the GI problem 

The known $k$-equivalent graphs were introduced in \cite{CFI} and \cite{Evdokimov99} as token counterexamples to the $k$-dim WL method. To the extend that this work establishes an extension to the WL method characterising these graphs, it removes the known counterexamples. However minor variations to each family of graphs can be trivially constructed while preserving the property of $k$-equivalence.

The decomposition method presented here is non-trivial in that it does not simply address the token counterexamples in isolation, but holds for all $k$-equivalent graphs for which the assumptions of Section \ref{sec:assumptions} are satisfied.

These assumptions are satisfied by the known $k$-equivalent graphs and minor trivial variants of such families.  Indeed, constructing counterexamples to this extension would require finding graphs with novel properties, for which these assumptions do not hold.

In particular, where $k \ge 3$ and $m \ge k$, counterexamples must belong to at least one of the following categories:
\begin{enumerate}
 \item Prime $m$-equivalent graphs that cannot be refined to their orbits by the $k$-dim WL method.
 \item Alternatively, prime $m$-equivalent graphs that cannot be characterised by the recursive $k$-dim WL method.
 \item Non-$k$-equivalent graphs that cannot be refined to their orbits by the $k$-dim WL method
\end{enumerate}
No such graphs have yet been found, hence a proof that this extended method does not solve GI would require finding graphs with novel properties. In particular, while there is no reason to suspect that these assumptions do hold for general graphs, they have been shown here to hold for all known counterexamples to the recursive $k$-dim WL method.

In relating these results to general $k$-equivalent graphs, we note that not much is known regarding possible general properties of such graphs, as the known cases were, as mentioned, introduced as token counterexamples, and with the exception of the work of \cite{Barghi09, Smith10} the properties of $k$-equivalent graphs have not been explored further.

\subsection{Open Questions}

As mentioned above, very little is known regarding possible general properties of $k$-equivalent graphs. Trivial $k$-equivalent graphs do exist, namely graphs which are $k$-isoregular (also known as $k$-tuple regular), however such graphs have been completely characterised for $k \ge 5$. One interesting open question relating to such graphs is whether $t$-isoregular graphs exist for $t \in \{3,4\}$ that cannot be characterised by the \emph{recursive} $3$-dim WL method. To find such a graph it would suffice to find either a non-vertex-transitive 3-isoregular graph or a 4-isoregular graph that is not distance-transitive, however to the knowledge of the author no such graphs are known to exist.

Apart from the consideration of $k$-isoregular graphs, what other kinds of potentially $k$-equivalent graphs exist? In particular, are there any general properties that such graphs must possess (in addition to $k$-equivalence) that restrict possible types? The two known families of $k$-equivalent graphs share several important properties, and in essence `obtain' their $k$-equivalence in the same way, via the expansion of graphs with large separator sizes. However these properties also make the graphs amenable to classification by the extended WL scheme constructed here.

One question relating to general properties of $k$-equivalent graphs was previously raised in \cite{Smith10}, in which the possibility that the $k$-dim WL method (for some bounded $k$) may suffice to distinguish pairs of strongly regular graphs was raised. It was noted that strongly regular graphs have particularly simple cellular closures (related to the WL method in Chapter \ref{sec:schemes}). In particular, the coherent configuration corresponding to a strongly regular graph with adjacency matrix $A$ has only three basis relations, $\{I, A, (J-I-A)\}$, where $I$ is the identity matrix and $J$ the all-1 matrix. An interesting open problem is an analysis of 
the $k$-extended cellular closures of strongly regular graphs, relating to the question of whether $k$-equivalent graphs can be strongly regular (for some bounded $k > 2$).

Related to the family of $k$-equivalent graphs are graphs which the $k$-dim WL method fails to refine to their orbits. General properties of such graphs are also unknown; trivial cases can formed via constructing a graph containing mutually CWS $k$-equivalent subgraphs, such as those belonging to $M'_k$, however do other graphs exist for which the recursive $k$-dim WL method fails?

Regarding the distinction between refining a graph to its orbits via the $k$-dim WL method, and determining the automorphism group by recursively stabilisation using the recursive $k$-dim WL method, the following question occurs.
If the $k$-dim WL method (for $k \ge 3$) refines a graph to its orbits, must the recursive $k$-dim WL method characterise the graph. In other words, will the vertex-stabilised graph also be refined to its orbits by the $k$-dim WL method?

Finally, for which graphs do the assumptions of Section \ref{sec:assumptions} \emph{not} hold?

\begin{spacing}{1}
\bibliographystyle{test_eprint}
\bibliography{references}

\begin{thebibliography}{10}

\bibitem{Alzaga10}
A.~Alzaga, R.~Iglesias, and R.~Pignol.
\newblock Spectra of symmetric powers of graphs and the weisfeiler-lehman
  refinements.
\newblock {\em J. Comb. Theory, Ser. B}, 100(6):671--682, 2010.

\bibitem{Arvind06}
V.~Arvind and Piyush~P. Kurur.
\newblock Graph isomorphism is in spp.
\newblock {\em Inf. Comput.}, 204(5):835--852, 2006.

\bibitem{Audenaert07}
K.~Audenaert, C.~Godsil, G.~Royle, and T.~Rudolph.
\newblock Symmetric squares of graphs.
\newblock {\em J. Comb. Theory Ser. B}, 97(1):74--90, 2007.

\bibitem{Babai80}
L.~Babai, P.~Erdos, and S.~M. Selkow.
\newblock Random graph isomorphism.
\newblock {\em SIAM J. Comput.}, 9(3):628--635, 1980.

\bibitem{Babai82}
L.~Babai, D.~Yu. Grigoryev, and L.~M. Mount.
\newblock Isomorphism of graphs with bounded eigenvalue multiplicity.
\newblock In {\em Proc. ACM STOC}, pages 310--324, New York, NY, USA, 1982.
  ACM.

\bibitem{Babai84}
L.~Babai, W.~M. Kantor, and E.~M. Luks.
\newblock Computational complexity and the classification of finite simple
  groups.
\newblock In {\em Proc. IEEE FOCS}, pages 162--171, New York, 1983. IEEE
  Computer Soc. Press.

\bibitem{Babai79}
L.~Babai and L.~Kucera.
\newblock Canonical labeling of graphs in linear average time.
\newblock In {\em 20th Ann. Sympos. Foundations Comput. Sci.}, pages 39--46,
  New York, 1979. IEEE Computer Soc. Press.

\bibitem{Babai83}
L.~Babai and E.~M. Luks.
\newblock Canonical labeling of graphs.
\newblock In {\em Proc. ACM STOC}, pages 171--183, New York, NY, USA, 1983.
  ACM.

\bibitem{Barghi09}
A.~R. Barghi and I.~Ponomarenko.
\newblock Non-isomorphic graphs with cospectral symmetric powers.
\newblock {\em Electr. J. Comb.}, 16(1), 2009.

\bibitem{CFI}
J.~Cai, M.~Furer, and N.~Immerman.
\newblock An optimal lower bound on the number of variables for graph
  identification.
\newblock {\em Combinatorica}, 12(4):389--410, 1992.

\bibitem{Cameron80}
P.~J. Cameron.
\newblock 6-transitive graphs.
\newblock {\em J. Comb. Theory, Ser. B}, 28(2):168--179, 1980.

\bibitem{Cameron03}
P.~J. Cameron.
\newblock Coherent configurations, association schemes and permutation groups.
\newblock In {\em Groups, Combinatorics and Geometry}, pages 55--72. World
  Scientific, 2003.

\bibitem{Colbourn81}
C.~J. Colbourn and K.~S. Booth.
\newblock Linear time automorphism algorithms for trees, interval graphs, and
  planar graphs.
\newblock {\em SIAM J. Comput.}, 10(1):203--225, 1981.

\bibitem{Douglas08}
B.~L. Douglas and J.~B. Wang.
\newblock A classical approach to the graph isomorphism problem using quantum
  walks.
\newblock {\em J. Phys. A}, 41(7):075303, 2008.

\bibitem{Evdokimov99}
S.~Evdokimov and I.~N. Ponomarenko.
\newblock On highly closed cellular algebras and highly closed isomorphisms.
\newblock {\em Electr. J. Comb.}, 6, 1999.

\bibitem{Friedland89}
S.~Friedland.
\newblock Coherent algebras and the graph isomorphism problem.
\newblock {\em Discrete Applied Mathematics}, 25(1-2):73--98, 1989.

\bibitem{Furer87}
M.~F\"urer.
\newblock A counterexample in graph isomorphism testing.
\newblock Technical report, 1987.

\bibitem{Gamble10}
J.~K. Gamble, M.~Friesen, D.~Zhou, R.~Joynt, and S.~N. Coppersmith.
\newblock Two-particle quantum walks applied to the graph isomorphism problem.
\newblock {\em Phys. Rev. A}, 81(5):052313, 2010.

\bibitem{Golfand78}
J.~J. Gol'fand and M.~H. Klin.
\newblock On $k$-homogeneous graphs (in russian).
\newblock {\em Algorithmic Investigations in Combinatorics}, pages 76--85,
  1978.

\bibitem{Hopcroft74}
J.~E. Hopcroft and J.~K. Wong.
\newblock Linear time algorithm for isomorphism of planar graphs (preliminary
  report).
\newblock In {\em Proc. ACM STOC}, pages 172--184, New York, NY, USA, 1974.
  ACM.

\bibitem{Koebler93}
J.~K{\"o}bler, U.~Sch{\"o}ning, and J.~Tor{\'a}n.
\newblock {\em The Graph Isomorphism Problem, Its Structural Complexity}.
\newblock Birkh{\"a}user, 1993.

\bibitem{Luks80}
E.~M. Luks.
\newblock Isomorphism of graphs of bounded valence can be tested in polynomial
  time.
\newblock In {\em Proc. IEEE FOCS}, pages 42--49, Washington, DC, USA, 1980.
  IEEE Computer Society.

\bibitem{Mathon79}
Rudolf M.
\newblock A note on the graph isomorphism counting problem.
\newblock {\em Inf. Process. Lett.}, 8(3):131--132, 1979.

\bibitem{McKay81}
B.~D. McKay.
\newblock Practical graph isomorphism.
\newblock {\em Congressus Numerantium}, 30:45--87, 1981.

\bibitem{Miller77}
G.~L. Miller.
\newblock Graph isomorphism, general remarks.
\newblock In {\em Proc. ACM STOC}, pages 143--150, New York, NY, USA, 1977.
  ACM.

\bibitem{Miller80}
G.~L. Miller.
\newblock Isomorphism testing for graphs of bounded genus.
\newblock In {\em Proc. ACM STOC}, pages 225--235, New York, NY, USA, 1980.
  ACM.

\bibitem{Miyazaki97}
T.~Miyazaki.
\newblock The complexity of {M}c{K}ay's canonical labeling algorithm.
\newblock In {\em Groups and computation, {II}}, volume~28 of {\em DIMACS Ser.
  Discrete Math. Theoret. Comput. Sci.}, pages 239--256. Amer. Math. Soc.,
  1997.

\bibitem{Pikhurko10}
O.~Pikhurko and O.~Verbitsky.
\newblock Logical complexity of graphs: a survey.
\newblock 2010.

\bibitem{Read77}
R.~C. Read and D.~G. Corneil.
\newblock The graph isomorphism disease.
\newblock {\em J. Graph Theory}, 1:339--363, 1977.

\bibitem{Schoning87}
U.~Sch\"{o}ning.
\newblock Graph isomorphism is in the low hierarchy.
\newblock In {\em 4th Annual Symposium on Theoretical Aspects of Computer
  Sciences on STACS 87}, pages 114--124, London, UK, 1987. Springer-Verlag.

\bibitem{Smith10}
J.~Smith.
\newblock k-boson quantum walks do not distinguish arbitrary graphs.
\newblock abs/1004.0206, 2010.

\bibitem{Zemlyachenko85}
N.~M.~Kornienko V.~M.~Zemlyachenko and R.~I. Tyshkevich.
\newblock Graph isomorphism problem.
\newblock {\em Journal of Soviet Mathematics}, 29:1426--1481, 1985.

\bibitem{Weisfeiler76}
B.~Weisfeiler, editor.
\newblock {\em On construction and identification of graphs}.
\newblock Lecture Notes in Mathematics, Vol. 558. Springer-Verlag, Berlin,
  1976.
\newblock With contributions by A. Lehman, G. M. Adelson-Velsky, V. Arlazarov,
  I. Faragev, A. Uskov, I. Zuev, M. Rosenfeld and B. Weisfeiler.

\bibitem{WL68}
B.~Weisfeiler and A.~A. Lehman.
\newblock A reduction of a graph to a canonical form and an algebra arising
  during this reduction (in russian).
\newblock {\em Nauchno-Technicheskaya Informatsia, Seriya 2}, 9:12--16, 1968.

\end{thebibliography}
\end{spacing}

\end{document}